\documentclass[11pt]{amsart} 
\usepackage{amsmath,amssymb,latexsym, amscd}
\usepackage[all]{xy}
\usepackage{todonotes}

\usepackage{hyperref}

\definecolor{dark_purple}{rgb}{0.4, 0.0, 0.4}
\definecolor{dark_green}{rgb}{0.0, 0.7, 0.0}

\def\H{{\mathcal H}}
\def\HH{{\underline{\mathcal H}}} 

\def\N{{\mathbb N}}

\def\C{{\mathbb C}}
\def\R{{\mathbb R}}
\def\Z{{\mathbb Z}}

\def\CC{\underline{\mathbb C}} 

\def\v{\varphi}
\def\a{\alpha}
\def\u{{\bf u}}

\def\z{\zeta}

\newcommand{\Ss}{\mathcal{S}} 

\newcommand{\Ker}[1]{\mathsf{Ker}~ }

\newcommand{\codim}[1]{\mathsf{codim}~ }

\renewcommand{\subset}{\subseteq} 

\newtheorem{theorem}{Theorem}[section]
\newtheorem{proposition}  [theorem]  {Proposition} 
\newtheorem{definition}   [theorem]  {Definition} 
\newtheorem{lemma}        [theorem]  {Lemma} 
\newtheorem{corollary}    [theorem]  {Corollary}
\newtheorem{example}      [theorem]  {Example} 
\newtheorem{remark}		  [theorem]  {Remark} 

\newcommand{\wt}{\widetilde}

\DeclareMathOperator{\image}{Im} 
\DeclareMathOperator{\U}{U} 
\DeclareMathOperator{\SU}{SU} 

\DeclareMathOperator{\spa}{span} 

\newcommand{\ii}{\mathrm{i}} 
\newcommand{\eu}{\mathrm{e}} 

\newcommand{\GL}{\mathrm{GL}}

\newcommand{\gl}{\mathfrak{gl}}

\newcommand{\CP}{{\mathbb C}P}
\newcommand{\RP}{{\mathbb R}P}

\newcommand{\pa}{\partial}
\newcommand{\dbar}{\overline{\partial}}
\newcommand{\zbar}{\bar{z}}

\newcommand{\ov}{\overline}

\newcommand{\ga}{\gamma}
\newcommand{\la}{\lambda}

\begin{document}

\title[Harmonic maps and shift-invariant subspaces]
{Harmonic maps and shift-invariant subspaces}
\author[A. Aleman]{Alexandru Aleman}
\address{Lund University, Mathematics, Faculty of Science, P.O. Box 118, S-221 00 Lund, Sweden}
\email{alexandru.aleman@math.lu.se}
\author[R. Pacheco]{Rui Pacheco}
\address{Centro de Matem\'{a}tica e Aplica\c{c}{\~{o}}es (CMA-UBI), Universidade da Beira Interior, 6201 -- 001
Covilh{\~{a}}, Portugal.}
\email{rpacheco@ubi.pt}
\thanks{The second author was partially supported by Funda\c{c}\~{a}o para a Ci\^{e}ncia e Tecnologia through the project UID/MAT/00212/2019.}
\author[J.C. Wood]{John C. Wood}
\address{School of Mathematics, University of Leeds, LS2 9JT, G.B.}
\email{j.c.wood@leeds.ac.uk}
\thanks{}


\maketitle
\begin{abstract} We investigate in detail the connection between harmonic maps from Riemann surfaces into the unitary group $\U(n)$ and their Grassmannian models: these are families of shift-invariant subspaces of $L^2(S^1,\C^n)$. With the help of operator-theoretic methods we derive a criterion for finiteness of the uniton number which has a large number of applications discussed in the paper.
\end{abstract}

\section{Summary of results} \label{sec:statement}

For harmonic maps from Riemann surfaces to $\U(n)$ we give some new criteria for finiteness of the uniton number involving an operator arising in the Grassmannian model, see Theorem \ref{finite uniton}.  We then give some applications including a criterion for finiteness of the uniton number for harmonic maps given by a constant \emph{potential} in the sense of \cite{DPW}, see Theorem \ref{constant potentials}, which extends a result in \cite{burstall-pedit}.  Some geometrical examples follow in \S  \ref{subsec:further} involving the \emph{first return map} which gives a simple proof that \emph{superconformal} harmonic maps are never of finite uniton number (Proposition \ref{prop:superconf} and Corollary \ref{cor:superm-sphere}).  We also give conditions for finiteness of the uniton number for more general diagrams (Theorem \ref{th:external}).
 On the way, we obtain a new  description of the Grassmannian model:  Proposition \ref{Grass-model} provides an intrinsic characterization of shift-invariant subspaces that can be represented by a smooth loop in $\U(n)$, and Theorem \ref{d-bar-invariance} generalizes this to subbundles of shift-invariant subspaces. These two results give an alternative to the approach in \cite[\S 7.1]{pressley-segal} and we hope they are of independent interest.   We thank Martin Guest for comments on this and for pointing out to us the interpretation of some of our constructions in terms of $D$-modules.

\section{Introduction and Preliminaries} \label{sec:intro}
This section is written with the needs of the functional analysis community in mind: much of it will be familiar to geometers working in harmonic maps.
A smooth map $\v$ between two Riemannian manifolds $(M,g)$ and $(N,h)$ is said to be \emph{harmonic} if it is a critical point of the energy functional
$$ E(\v, D)=\frac{1}{2}\int_D|d\v|^2\omega_g$$
for any relatively compact $D$ in $M$, where $\omega_g$ is the volume measure, and $|d\v|^2 $ is the Hilbert-Schmidt  norm of the differential  of $\v$.
Using coordinates it is easy to see that
$E(\v, D)$ is the natural  generalization of the classical integral of Dirichlet in $\mathbb{R}^d$:
$$
D[u]=\int_\Omega |\nabla u|^2\, d V, \qquad \Omega\subset \R^d.
$$

Recall that a \emph{Riemann surface} is a $2$-dimensional manifold equipped with \emph{complex charts} $M \supseteq V \to U \subseteq \C$ such that any two charts are related by a complex analytic function between open subsets of $\C$, see, for example, \cite{forster}.  The image in $U$ of a point $z \in V \subseteq M$ under a complex chart is called its \emph{(complex) coordinate} (with respect to that chart); as is normal we shall denote this by the same letter $z$. Because of conformal invariance, see, for example, \cite[\S 1.2]{wood-60}, the above energy functional, and so the notion of \emph{harmonic map from a Riemann surface}, is well defined.

In this paper we consider harmonic maps from a Riemann surface $M$ into the group $\U(n)$ of unitary matrices of order $n$.  Such maps are interesting for many reasons. For example (see \cite{svensson-wood-unitons}), when $M$ coincides with the Riemann sphere $S^2$,  they are equivalent to harmonic maps of finite energy from the plane,  they  provide a nonlinear $\sigma$-model for particle physics \cite{zakrzewski}, and  they give minimal branched immersions  in the sense of \cite{GOR}. As any compact Lie group has a totally geodesic embedding into $\U(n)$ for some $n$, these maps include harmonic maps into compact Lie groups and, in particular, harmonic maps into symmetric spaces. There is a vast literature on the subject (for example, see \cite{eells-lemaire, urakawa} for the general theory and \cite{svensson-wood-unitons, wood-60} for some background relevant to this paper).

Given a smooth map $\v:M\to \U(n)$,  one  can use  a standard variational argument to derive an equivalent differential condition for harmonicity. More precisely (see \cite{uhlenbeck} or \cite[Sec.~3.1]{svensson-wood-unitons}), consider the matrix-valued 1-form
\begin{equation}\label{Aphi}
\tfrac{1}{2}\v^{-1}d\v := A^\v_z d z+A^\v_{\bar{z}}d\bar{z},
\end{equation}
where $z$ is a  local  (complex) coordinate on $M$. Then it turns out that $\v$  is harmonic if and only if
\begin{equation}\label{harmcond} (A^\v_z)_{\bar{z}}+(A^\v_{\bar{z}})_z=0.\end{equation}

This equation, as for all our key equations below, is invariant under change of local complex coordinate, an extension of the usual conformal invariance of the Laplacian.

We remark that, since the right-hand side of \eqref{Aphi} has values in the Lie algebra $\u(n)$ of $\U(n)$, $A^\v_{\bar{z}}$ is minus the adjoint of $A^\v_z$.  Also, computing directly from $A^\v_z = \frac{1}{2}\v^{-1}\v_z$ and $A^\v_{\bar{z}} = \frac{1}{2}\v^{-1}\v_{\bar{z}}$ we obtain a further equation called the \emph{integrability equation},
\begin{equation}\label{integrability}
(A^\v_z)_{\bar{z}}-(A^\v_{\bar{z}})_z= 2[A^\v_z, A^\v_{\bar{z}}],
\end{equation}
valid for any smooth map $\v$.

In her study of harmonic maps \cite{uhlenbeck}, K.~Uhlenbeck introduced the notion of an \emph{extended solution}, which is a smooth map $\Phi:S^1\times M\to \U(n)$ satisfying $\Phi(1,\cdot)=I$ and such that, for every local coordinate $z$ on $M$, there are $\gl(n,\C)$-valued maps $A_z$ and $A_{\bar{z}}$ for which
\begin{equation}\label{extsol}
\Phi(\la,\cdot)^{-1}d\Phi(\la,\cdot)=(1-\la^{-1})A_z d z+(1-\la)A_{\bar{z}}d\bar{z}.
\end{equation}
The introduction of a `spectral parameter' $\la$ is a key idea in integrable systems, see \cite{guest-book} for background and references.  Note that we can consider $\Phi$ as a map from $M$ into the \emph{loop group} of $\U(n)$ defined by
$\Omega\U(n) = \{\gamma:S^1 \to \U(n) \text{ smooth}: \gamma(1) = I\}$. If $\Phi$ is an extended solution, then
 $\v=\Phi(-1,\cdot)$ is a harmonic map with $A^\v_z=A_z$ and $A^\v_{\bar{z}}=A_{\bar{z}}$.
 Conversely,  for a given harmonic map $\v:M\to \U(n)$, an extended solution with the property that
$$
\Phi^{-1}(\la,\cdot)d\Phi(\la,\cdot)=(1-\la^{-1})A^\v_zd z+(1-\la)A^\v_{\bar{z}}d\bar{z}
$$
is said to be \emph{associated} to $\v$, and we have
$$
\Phi(-1,\cdot)=u\v
$$
for some constant $u\in \U(n)$.
If $M$ is simply connected, the existence of extended solutions is guaranteed  by (and is actually equivalent to, see \cite{uhlenbeck}) the harmonicity condition  \eqref{harmcond}. The solution is unique up to multiplication from the left by a constant loop, i.e., a $\U(n)$-valued function on $S^1$, independent of $z\in M$.
Moreover (see \cite[Thm 2.2]{uhlenbeck} and Remark \ref{rem:Grauert}), the extended solution can be chosen to be a smooth map, or even holomorphic in $\lambda\in \C\setminus\{0\}$ and real analytic in $M$.

We shall make use of a method
called the \emph{Grassmannian model} \cite{segal}, which associates to an extended solution $\Phi$ the family of closed subspaces $W(z),~z\in M,$   of  $L^2(S^1,\C^n)$, defined by
\begin{equation}\label{W-def}
W(z)=\Phi(\cdot,z)\H_+,
\end{equation}
where  $\H_+$ is the usual Hardy space of $\C^n$-valued functions, i.e., the closed subspace of  $\H := L^2(S^1,\C^n)$ consisting of Fourier series whose negative coefficients vanish.
 Note that the subspaces $W(z)$ form the fibres of a smooth bundle $W$ over the Riemann surface (which is, in fact, a \emph{subbundle} of the \emph{trivial bundle}
$\HH := M \times L^2(S^1,\C^n)$, see \S \ref{subsec:grass-model}.
Let us denote by $\partial_z$ and $\partial_{\bar{z}}$ the derivatives with respect to  $z$ and $\bar{z}$ on $M$, respectively, and by $S$ the forward shift on $L^2(S^1,\C^n)$:
$$(Sf)(\la)=\la f(\la), \qquad \la\in S^1.$$
If $f:S^1\times M\to\C^n$ is differentiable in the second variable and satisfies $f(\cdot,z)\in W(z),~z\in M$, it follows from \eqref{extsol} that
\begin{equation}\label{W-eq}
 S\partial_zf(\cdot,z)\in W(z),  \quad \partial_{\bar{z}}f(\cdot,z)\in W(z),  \end{equation} i.e., in terms of  differentiable sections we have
\begin{equation}\label{propW}
S\partial_zW(z)\subset W(z),\quad \partial_{\bar{z}}W(z)\subset W(z),
\end{equation}
which we shall often abbreviate to $S\partial_zW\subset W$ and
$\partial_{\bar{z}}W\subset W$.
In fact, these equations are equivalent to \eqref{extsol} see \cite{segal, guest-book}.
An obvious  advantage of this approach is that it yields coordinate-free equations, since the derivatives of the coordinates involved in the chain rule  are absorbed in the corresponding subspaces. On the other hand, the objects involved here are shift-invariant subspaces of $L^2(S^1,\C^n)$, and using them in order to develop a qualitative theory of the solutions of \eqref{extsol} leads to a number of challenging problems, due to the fact that, in general,  Uhlenbeck's extended solutions cannot be found explicitly.
Moreover, alternative approaches to the theory of harmonic maps lead to  more general representations of
 the bundles $W$.  For example, \emph{extended framings} associated to primitive harmonic maps into $k$-symmetric spaces (see \cite{aleman-pacheco-wood-symmetry}),
or the Dorfmeister--Pedit--Wu construction \cite{DPW}, lead naturally to  bundles  $W$, satisfying   \eqref{W-eq}, and  given by
\begin{equation}\label{gen-W}W(z)=\psi(\cdot, z)\H_+,\end{equation}
where $\psi:S^1\times M\to \GL(n,\C)$ is a smooth map.
The bundles obtained this way are characterized in \cite[\S 7.1]{pressley-segal};  see Theorem \ref{d-bar-invariance} for an alternative approach.
Of course, \eqref{W-eq} implies that the negative Fourier coefficients of
$\lambda\psi^{-1}(\la,z)\partial_z\psi(\la,z)$ and of  $\psi^{-1}(\la,z)\partial_{\bar{z}}\psi(\la,z)$ vanish, but, in general, it is not clear which maps $\psi$ occur  this way. For example (see  \cite{DPW}), this condition is fulfilled if the \emph{potential} $\psi^{-1}(\la,z)\partial_z\psi(\la,z)$ is holomorphic in $z\in M$.

The Iwasawa decomposition of loop groups  \cite[Theorem (8.1.1)]{pressley-segal} implies that $W(z)=\Phi(\cdot,z)\H_+$, with $\Phi:S^1\times M\to \U(n)$ smooth; given such a $\Phi$, it is easy to verify that  \eqref{W-eq} implies that $\Phi\,\Phi^{-1}\!(1,\cdot)$ is an extended solution.  However, $\Phi$ cannot usually  be determined explicitly  except in the finite uniton case, cf.\ \cite{svensson-wood-unitons}, and in many cases we
would like to develop our study based on the properties of $\psi$ or $W$.

The main purpose of this paper is to explore the connection between harmonic maps which possess extended solutions, and the associated infinite-dimensional family (i.e., bundle) $W = W(z)$ of shift-invariant subspaces
\eqref{W-def}.
By extension we shall call the family $W(z)$ an \emph{extended solution} as well.
 It turns out that this point of view leads to very interesting new results about harmonic maps  and  to a beautiful interplay between differential geometry and operator theory.  Recall that all harmonic maps from a simply connected surface have an extended solution locally, and globally on a simply connected domain; for arbitrary domains any two extended solutions, and so the resulting $W$, differ by premultiplication by a loop which does not change the criteria in our main results.

In this paper, we shall focus on the important question of finiteness of the uniton number; in \cite{aleman-pacheco-wood-symmetry}, we take up the question of \emph{symmetry}.  We say  that the  harmonic map $\v:M\to \U(n)$ has \emph{finite uniton number} if there exists an extended solution $\Phi_0$  associated to $\v$, \emph{which is defined on the whole of $M$} and is a trigonometric polynomial in $\lambda\in S^1$, i.e.,
there exist $r,s\in \N$ such that
\begin{equation} \label{finite-uniton-no}
\Phi_0(\la,z)=\sum_{k=-r}^s C_k(z)\la^k, \quad C_k:M \to \gl(n,\C) \text{ smooth}.
\end{equation}
Consequently, an  arbitrary extended solution $\Phi$  associated to $\v$ will have the form $v\Phi_0$ with $v:S^1\to \U(n)$ independent of $z\in M$ (and
satisfying $v(1)=v(-1)=I$);  we shall say that such a $\Phi$, and the corresponding $W=\Phi\H_+$, are of finite uniton number, too.
The standard factorization theory of matrix-valued functions on  $S^1$ (see \cite{Nikolskii,Peller}) shows that such functions are essentially polynomial  Blaschke--Potapov products depending on $z\in M$.  More precisely,  there exist  subbundles $\alpha_1,\ldots,\alpha_m$ of $M\times\C^n$  such that
\begin{equation}\label{bp}
\Phi(\la,z)=v(\la)\prod_{j=1}^{m}(\pi_{\a_j(z)}+\la\pi_{\a_j(z)}^\perp),
\end{equation}
where $\pi_\a$ denotes the orthogonal (Hermitian) projection onto the subspace $\a$ of $\C^n$ and $\pi_\a^\perp=\pi_{\a^\perp}$.
It is well known that, for  $n>1$,  the factors in the product above are not necessarily unique, and this  is one of the major technical  difficulties in dealing with these objects. However, it is shown in \cite{uhlenbeck} that the factors can be chosen such that for $k\le m$, the partial products
$$
\Phi_k(\la,z)=\prod_{j=1}^{k}(\pi_{\a_j(z)}+\la\pi_{\a_j(z)}^\perp)
$$
are  extended solutions as well. The Blaschke--Potapov factors involved in such a factorization are called {\it unitons} of   $\Phi$ and \eqref{bp} is its {\it uniton factorization}.
For $\la=-1$ this yields a factorization of the original harmonic map $\v$, also called a uniton factorization,  and the corresponding factors are called \emph{unitons of\/ $\v$}.  Not every factorization of the Blaschke--Potapov product above yields unitons: a complete description of those which do yield unitons  in terms of `$\la$-filtrations' has been obtained in \cite{svensson-wood-unitons}.

As a special case we say that $\Phi$ and the corresponding $W = \Phi\H_+$ is \emph{$S^1$-invariant} if $f \in W$ implies that $f_\mu \in W$ for all $\mu \in S^1$, where $f_\mu(\lambda)=f(\mu\lambda)$, \ $\la \in S^1$. As in \cite[Proposition 10.4]{uhlenbeck}, an $S^1$-invariant extended solution has a uniton factorization with nested unitons $\alpha_1 \subset \alpha_2 \subset \cdots \subset \alpha_m$, which are \emph{holomorphic}, i.e., $\pa_{\bar{z}} (\alpha_i) \subseteq \alpha_i$, and \emph{superhorizontal},
i.e.,
\begin{equation} \label{superhor}
\pa_z(\alpha_i) \subseteq \alpha_{i+1}, \quad i=1,\ldots, m-1,
\end{equation}
cf.\ \cite[\S 2.3, 3.3]{svensson-wood-unitons}.

One of the remarkable results in \cite{uhlenbeck} is that, for the Riemann sphere $M=S^2$, every harmonic map has finite uniton number.  More generally, on a \emph{compact} Riemann surface $M$, any extended solution has the form \eqref{bp} \cite{OV}, so any harmonic map \emph{which has an extended solution on $M$} is of finite uniton number.
By the uniformization (or Riemann mapping) theorem \cite{forster}, all Riemann surfaces $M$ have a \emph{simply connected covering space} $\Pi:\tilde{M} \to M$ where $\tilde{M}$ is the Riemann surface $S^2$, $\C$ or the open unit disk of $\C$.
The projection $\Pi$ is holomorphic so that a map $\v$ from $M$ is harmonic if and only if $\v \circ \Pi$ is harmonic. \emph{All} harmonic maps have an extended solution on a simply connected covering space, for example, harmonic maps from a torus can be thought of as doubly periodic harmonic maps on $\C$, and since $\C$ is simply connected, have extended solutions on $\C$.
If there is an extended solution which is doubly periodic too, and so descends to the torus, then, by the above result for compact surfaces, the harmonic map is of finite uniton number; otherwise it is not, as in the case of Clifford solutions,
cf.\ Example \ref{ex:Clifford}.  In fact, we show in \S \ref{subsec:applns ha maps} that the notion of finite uniton number is \emph{local}: a harmonic map $\v$ from a Riemann surface is of finite uniton number if and only if its restriction to an open set is of finite uniton number, and this holds if and only if its composition with $\Pi$ is of finite uniton number.  We should also point out that on non-compact Riemann surfaces there always exist harmonic maps which are not of finite uniton number
(see  \S  \ref{subsec:constant potentials} below).

Motivated by this discussion, our main result in this direction provides a necessary and sufficient criterion for finiteness of the uniton number which will be proved in \S  \ref{subsec:the criterion}, Theorem \ref{finite uniton}. At the level of the Grassmannian model  the criterion  is easy to explain: it simply requires that the sequence $$W_{(i)}:=W+\partial_z W+\cdots+\partial^i_zW$$
stabilizes.  This sequence is a key tool in our approach. In \S \ref{subsec:the Gauss sequence} we prove that each $W_{(i)}$ extends to a holomorphic subbundle of the trivial bundle $\HH=M\times L^2(S^1,\C^n)$. If $W=\psi\H_+$ as in \eqref{gen-W}, the result can be rephrased in terms of the differential operators $T_\psi(\la)=\partial_z+ \psi^{-1}\partial_z\psi$ and it asserts that the uniton number is finite if and only if the maximum power of $\lambda^{-1}$ in $T_\psi^i(\lambda)$ remains bounded for $i\in \mathbb{N}$.
 It turns out that this general criterion has a number of interesting applications.  For example, in certain extreme cases we can obtain necessary and sufficient conditions for the uniton number to be finite using only the bundle $\image{A_z^\varphi} := A_z^\varphi\C^n$ and its $\partial_z$-derivative (see \S \ref{subsec:direct applications} for a more general version). In particular, this will lead to a complete characterization of such harmonic maps into $\U(2)$ (Corollary \ref{U(2)}). Another  important special case where the criterion can be successfully applied is that of \emph{constant potentials}, i.e. when $W=\psi\H_+$ as in \eqref{gen-W}, with $\psi^{-1}\partial_z\psi$ constant in $z\in M$.  In Theorem \ref{constant potentials} we prove that the underlying harmonic map $\varphi$ has finite uniton number if and only if $$(\lambda,\mu)\mapsto \det(\mu I-\psi^{-1}\partial_z\psi)$$ is holomorphic in $\{(\la,\mu)\in \C^2:~|\la|<1\}$. This extends   to $\U(n)$ a result proved for $\SU(2)$  by F.E.~Burstall and F.~Pedit \cite{burstall-pedit}.

In \S \ref{subsec:further}, we take a more geometrical approach, interpreting the operator $T_{\varphi}(\lambda)$ in terms of $A^{\v}_z$ and a derivation $D^{\v}_z$ associated to $\v$.  This is particularly effective for harmonic maps into Grassmannians.  Iterating a construction akin to taking a Gauss map generates a diagram of subbundles and maps between them.  If this diagram terminates, then the harmonic map is of finite uniton number (Proposition \ref{prop:str-iso}).  Otherwise, it closes up and the composition of maps around the resulting cycle is called \emph{the first return map}. The nilpotency of this map gives a necessary condition for the uniton number to be finite (Theorem \ref{th:first-return}).  In particular, an interesting class of harmonic maps called \emph{superconformal}, which include the famous Clifford solutions on $\C$ and associated Clifford tori, are easily seen to be not of finite uniton number.

For more complicated diagrams of subbundles and maps, we generalize the above work to give conditions for finiteness of the uniton number in terms of the composition of maps around cycles, see Theorem \ref{th:external}.

\section{The bounded powers criterion}

 In this section we establish a criterion for whether an extended solution is of finite uniton number or not.  We then give some applications.
 First we introduce a description of the Grassmannian model, which is the main tool used in our approach.

\subsection{The Grassmannian model} \label{subsec:grass-model}

 We recall the Grassmannian model of Pressley and Segal \cite[\S 7 \& 8]{pressley-segal}: Let $\H_-=\H_+^\perp$, and let $P_\pm$ denote the orthogonal projections of $\H=L^2(S^1,\C^n)$ onto $\H_\pm$.
Set $\mathrm{Gr}^{(n)}$ equal to the Hilbert manifold of closed shift-invariant subspaces $V$ of $\H$ such that
$P_+:V\to \H_+$ is a Fredholm operator and $P_{-}:V\to \H_{-}$ is a Hilbert-Schmidt operator and set
$\mathrm{Gr}_{\infty}^{(n)}$ equal to the subspace of $\mathrm{Gr}^{(n)}$ satisfying the smoothness condition (iii) below.
Then Pressley and Segal show that the map $\gamma \mapsto \gamma\H_+$ is a bijection from the smooth loop group $\Omega \U(n)$ to $\mathrm{Gr}_{\infty}^{(n)}$. We here replace the Fredholm--Hilbert-Schmidt condition with the following simpler condition:
\begin{center}
 $S|V$ is a \emph{pure shift}, i.e. $\bigcap_{k \in \N} S^k V = \{0\}$, \\
of \emph{multiplicity} $n$, that is $\dim V\ominus SV=n.$
\end{center}
Here $V\ominus SV=V\cap (SV)^\perp$. Note that this condition holds if and only if $V$ has the form $V=\gamma \H_+$ with $\gamma(\la)\in \U(n)$ a.e. on $S^1$; the `if' direction is trivial, and the `only if' direction follows from
\cite[Lecture VI]{He}).

The following gives conditions for smoothness of $\gamma$.
\begin{proposition} \label{Grass-model}
Let $V$ be a closed shift-invariant subspace of $\H= L^2(S^1,\C^n)$,
such that $S|V$ is a pure shift of multiplicity $n$. Then
the following are equivalent $:$
\begin{enumerate}
	\item[(i)]  $V=\gamma \H_+$ for a unique $\gamma\in  \Omega\U(n);$
	\item[(ii)] $V\ominus SV$ is a subspace of $C^\infty(S^1,\C^n)$ of dimension  $n;$
\item[(iii)] The subspaces $P_+V^\perp$ and $P_-V$ are contained in $C^\infty(S^1,\C^n)$.
\end{enumerate}
\end{proposition}
\begin{proof} 	
	(i) $\Rightarrow$ (ii) is immediate, since $V\ominus SV=\{\gamma e:~e\in \C^n\}$. The reverse implication (ii) $\Rightarrow$ (i) is just the `constant version' of first part of the following theorem, so does not need a separate proof. 	
	To see that (i) $\Rightarrow$ (iii) note first that $V^\perp=\gamma\H_-$. Moreover, if $f\in \H_-$, then $\gamma(\la)f\in \H_-$, i.e. $P_+\gamma(\la)f=0$, for all $\la\in S^1$. Thus by the well known formula for $P_+$ (see \cite{Nikolskii}), we have	$$P_+\gamma f(\lambda)=P_+[(\gamma-\gamma(\lambda))f](\lambda)=\int_{S^1}\frac{(\gamma(\z)-\gamma(\la))f(\z)}{\z-\la}\frac{d\z}{2\pi i}.$$
Clearly, the integral on the right belongs to $C^\infty(S^1,\C^n)$. The proof that $P_-V\subset C^\infty(S^1,\C^n)$ is identical. Finally, if (iii) holds, $V=\gamma \H_+$ with $\gamma(\la)\in \U(n)$ a.e., in order to prove (i) it suffices to show that $\gamma\in C^\infty(S^1,\C^n)$. For every $e\in \C^n$ we have $P_-\gamma e,~P_+S^{-1}\gamma e\in C^\infty(S^1,\C^n)$. If $\gamma(\la)=\sum_{k\in \mathbb{Z}}\hat{\gamma}_k\la^k$, a direct computation  gives 	
$$P_+S^{-1}\gamma e=\la^{-1}(P_+\gamma e(\la)-\hat{\gamma}_0).$$
Thus $$\gamma e(\la)=P_-\gamma e(\la)+ P_+\gamma e(\la)= P_-\gamma e(\la)+ \la P_+S^{-1}\gamma e(\la)+\hat{\gamma}_0\in C^\infty(S^1,\C^n).$$
	 \end{proof}

\begin{remark}
{\rm (a) Proposition \ref{Grass-model} gives the following description: \emph{$\mathrm{Gr}_{\infty}^{(n)}$ is the set of closed shift-invariant subspaces $V$ of $\H$ with $S|V$ a pure shift of multiplicity $n$ such that one of the equivalent conditions (i), (ii) or (iii) in the proposition holds.}}

{\rm (b) From the proof of \cite[Theorem 8.3.2]{pressley-segal}, we see that the condition of $S|V$ being a pure shift of multiplicity $n$ is actually implied by the Fredholm--Hilbert-Schmidt condition. The reverse implication fails in general, but under the smoothness assumptions (ii) or (iii)  of Proposition \ref{Grass-model} the conditions are equivalent.}
\end{remark}

\begin{theorem}
	\label{d-bar-invariance} Let $V(z),~z\in M$ be a family of closed, shift-invariant subspaces of $L^2(S^1,\C^n)$, such that:
\begin{enumerate}	
	\item[(i)] $S|V(z)$ is a pure shift for all $z\in M$;
	\item[(ii)] the subspaces $$V(z)\ominus SV(z)=V(z)\cap (SV(z))^\perp$$ form a smooth subbundle of rank  $n$ of
	$M \times C^\infty(S^1,\C^n)$.
\end{enumerate}	
	 Then there exists a smooth $\U(n)$-valued function  $\psi$ on $S^1\times M$ with $\psi(1,z) = I$ such that $V(z)=\psi(\cdot,z)\H_+$. If, in addition,
\begin{enumerate}
	\item[(iii)] $\partial_{\bar{z}} V\subset V$,
\end{enumerate}
 then locally at every point of $M$, $V(z)=u(\cdot,z)\H_+$, with $u$ $\GL(n,\C)$-valued, smooth on $S^1 \times M$,  and holomorphic in $z\in M$. In particular, $V$ is a holomorphic subbundle of $\HH= M \times L^2(S^1,\C^n)$.
\end{theorem}

\begin{proof} Orthonormalize a locally smooth basis of $V(z)\ominus SV(z)$ to obtain the columns of a locally defined, smooth $\U(n)$-valued  function $v$ such that $V(z)\ominus SV(z)=v(\cdot, z)\C^n$. By (i) and the Wold--Kolmogorov decomposition \cite[p.~17]{Nikolskii}, we have $$V(z)=\bigoplus_{k\ge 0}S^k(V(z)\ominus SV(z))=
	\bigoplus_{k\ge 0}S^kv(\cdot, z)\C^n=v(\cdot,z)\H_+.$$
	It is also well known that the representation $V=v\H_+$ is unique up to multiplication from the right by a $\U(n)$-valued  function constant in $\la\in S^1$ (see \cite[p.~17]{Nikolskii}, for example). Therefore, for $v$ as above,  the function $\psi(\la,z)=v(\la,z)v^{-1}(1,z)$ is well defined, smooth on $S^1 \times M$, has $\psi(1,z) = I$, and satisfies $V(z)=\psi(\cdot,z)\H_+$.

In order to verify the second assertion, we shall prove that there exist locally  defined (in $z\in M$) $\GL(n,\C)$-valued smooth functions $F$ with $F(\cdot,z)e\in \H_+,~e\in \C^n$,   such that $\psi F$ is holomorphic in $z$.

 This gives the desired local representation with $u=\psi F$. For the last assertion, note first that from this representation it follows that  $V$ is a holomorphic Hilbert bundle on $M$ (see \cite[p.~43]{Lang}). Moreover, $V(z)$ has the locally holomorphic Riesz basis \cite[p.~132]{Nikolskii} $\{S^ku(\cdot,z)e_j:~1\le j\le n,~k\ge 0\}$, where $\{e_j:~1\le j\le n\}$ is the canonical basis in $\C^n$, which can be completed to the locally holomorphic Riesz basis of $L^2(S^1,\C^n)$ given
by $\{S^ku(\cdot,z)e_j:~1\le j\le n,~k\in \mathbb{Z}\}$; thus $V$ is a holomorphic subbundle of $\HH$.

We transfer the problem of finding a function $F$ as above to the complex plane by using a local coordinate which we also denote by $z \in U \subset \C$.
We now choose an open disc $\Delta=\{|z-z_0|<r\},~r>0$ in $U$; the condition that $z\to \psi(\la,  z) F(\la, z)$ is holomorphic in $\Delta$
is $$\partial_{\bar z}\bigl(\psi(\la,  z) F(\la, z)\bigr)=0.$$ If we denote by
 $$D_\psi(\la,z)=\psi^{-1}(\la,  z)\partial_{\bar z}\psi(\la,  z)$$ this becomes equivalent to the equation
\begin{equation}\label{original-d-bar-eq}\partial_{\bar z}F(\cdot, z)=-D_\psi(\cdot,z) F(\cdot,z).\end{equation}
We seek a smooth, $\GL(n,\C)$-valued solution  $F$   of  \eqref{original-d-bar-eq}  such that, for every $z \in \Delta$, $F(\cdot,z)$ belongs to $\Lambda^+(\gl(n,\C))$,  the closed subspace of  $L^2(S^1,\gl(n,\C))$ consisting of functions whose negative Fourier coefficients vanish.

The equation \eqref{original-d-bar-eq} can be handled with help of an integral equation which we shall introduce  below.   \\
Assume that the disc $\{|z-z_0|\le 2r\}$ lies in $U$, and let
 $\chi:\C\to [0,1]$ be a smooth function with $\chi(z)=1$ when $|z|\le 1$, $\chi(z)=0$ when $|z|\ge\frac4{3}$. Set $\chi_r(z)=\chi(\frac{z-z_0}{r})$.  Given a continuous,  $\gl(n,\C)$-valued function $F$ on  $\{|z-z_0|<2r\}$,  extend the product $\chi_rD_\psi F$ to be zero outside $\{|z-z_0|<\frac{4r}{3}\}$, and define
\begin{equation}\label{c-delta}C_\Delta F= \frac1{\pi}\int_\C\frac{\chi_r D_\psi F(t)}{t-z}dm(t)= \frac1{2\pi}\int_\C\frac{\chi_r D_\psi F(t+z)}{t}dm(t),\end{equation}
where $m$ denotes the area measure on $\C$.
Note that, if $G_0\in \gl(n,\C)$ is constant, a smooth solution $F$ of \begin{equation}\label{c-transform-eq} F+C_\Delta F=G_0\end{equation} will satisfy  \eqref{original-d-bar-eq}. Indeed, the basic property of the Cauchy transform:
$$\partial_{\bar z} \frac1{\pi}\int_\C \frac{H(t)}{t-z}dm(t)=\partial_{\bar z}\frac1{2\pi i}\int_\C \frac{H(t)}{t-z}dt\wedge d\overline{t} =H(z),$$
valid for every compactly supported continuous function $H$ on $\C$  (see for example, \cite{bell}),  implies that  on $\Delta$ we have
$$\partial_{\bar z}F+ D_\psi F=\partial_{\bar z}F+\chi_rD_\psi  F=
\partial_{\bar z}(F+C_\Delta F)= 0.$$

The remainder of the proof consists of two steps. We show that, if the radius $r$ of $\Delta$ is sufficiently small,
more precisely, if

\begin{equation}\label{size-of-r}  \sup_{\la\in S^1, |z-z_0|\le \frac4{3}}\|D_\psi(\la,z)\|\frac1{\pi}\int_{|t|<\frac4{3}r}\frac{dm(t)}{|t|}<\frac1{3},\end{equation}
then:\\
\emph{1) for every constant $G_0\in \gl(n,\C)$,  \eqref{c-transform-eq} has a smooth solution $F$ on $S^1 \times M$, such that, for each $z \in \Delta$, $F(z,\cdot)$  belongs to $\Lambda^+(\gl(n,\C))$;\\
2) for $G_0=I_{\C^n}$,  the solution of  \eqref{c-transform-eq} is $\GL(n,\C)$-valued. }\\

1) Let $\Delta'=\{|z-z_0|\le 2r\}\subset U$, and consider  for $k\in \mathbb{N}\cup\{0\}$  the Banach algebras $\mathcal{B}_k$ consisting of functions $F\in C^k(S^1\times\Delta')$ with the property that $F(\cdot,z)\in \Lambda^+(\gl(n,\C))$ for all $z\in \Delta'$,  endowed with the natural norm inherited from $C^k(S^1\times \Delta')$.\\
The  assumption (iii) in the statement together with the fact that  $\Delta'\subset U$  imply that $D_\psi\in \mathcal{B}_k$ for all $k$.  From the second equality in \eqref{c-delta} it  follows easily that $C_\Delta$ defines a bounded operator on each of these algebras (Banach spaces).\\
Our first claim is that for all $k\ge 0$, $C_\Delta|\mathcal{B}_k$ is compact.
   To this end, notice that it will be sufficient to prove the assertion when $k=0$, since the general case follows   by a repeated application of this result using  the second equality in \eqref{c-delta}.   Now observe that, for any $h\in L^1(\{|z-z_0|<4r\})$, the operator $C_h$, defined on $\mathcal{B}_0$ by
$$C_hF=\frac1{\pi}\int_\C(\chi_rD_\psi F(t))h(t-z)dm(t),$$
maps $\mathcal{B}_0$ into itself with
\begin{equation*}
\|C_h F\|_{\mathcal{B}_0}\le\frac1{\pi} \|h\|_1 \|D_\psi\|_{\mathcal{B}_0}\|F\|_{\mathcal{B}_0}.\end{equation*}

Let $(h_j)$  be  a sequence of polynomials in $z,\bar z$ converging in $ L^1(\{|z-z_0|<4r\})$ to the function $z\mapsto z^{-1}$. Obviously, $h_j(t-z)$ can be written as a finite sum
$$h_j(t-z)=\sum_k p_k^j(t)q_k^j(z),$$
with $p_k^j(s),q_k^n(s)$   polynomials in $s,\overline{s}$,  which  shows that the operators $C_{h_j}$ are  of finite rank and converge in operator norm to $C_\Delta$ on $\mathcal{B}_0$, hence $C_\Delta$ is compact.\\
Thus, $C_\Delta|\mathcal{B}_k$ is compact, and by  \eqref{size-of-r},  $I_{\mathcal{B}_0}+C_\Delta|\mathcal{B}_0$ is invertible. This implies that  $I_{\mathcal{B}_k}+C_\Delta|\mathcal{B}_k$ is invertible on $\mathcal{B}_k$ for every $k\ge 0$. Indeed,  from the inclusion $\mathcal{B}_k\subset\mathcal{B}_0,~k>0$, it follows that $I_{\mathcal{B}_k}+C_\Delta|\mathcal{B}_k$ is injective for all $k\in \mathbb{N}\cup\{0\}$.
Then, by the spectral theory of compact operators \cite[p.~238]{Lax}, we have that   $I_{\mathcal{B}_k}+C_\Delta|{\mathcal{B}_k}$ is invertible for all $k\ge 0$.  Thus, for each $k\in \mathbb{N}\cup\{0\}$,  \eqref{c-transform-eq} has a unique solution $F_k\in  \mathcal{B}_k$. Using again the inclusion
$\mathcal{B}_k\subset\mathcal{B}_0,~k>0$, we conclude that $F_k=F_0$, for all $k\ge 0$, i.e. $F=F_0$ is a  smooth solution of
\eqref{c-transform-eq}, such that $F(z,\cdot)$  belongs to $\Lambda^+(\gl(n,\C))$.

2) Assume that   $G_0=I_{\C^n}$,  and let $F$ be the smooth solution   of \eqref{c-transform-eq} given above. We want to show that   then $F(\lambda,z)\in
\GL(n,\C)$. Recall that if $k=0$, the operator norm of $C_\Delta|\mathcal{B}_0$ is less than $\frac1{3}$, by \eqref{size-of-r}.
 Our solution $F$  is  given by $F=\sum_{j=0}^\infty (-1)^jC^j_\Delta I_{\C^n}$, hence $$\|F\|_{\mathcal{B}_0}\le  \sum_{j=0}^\infty\|C_\Delta I_{\C^n}\|^j_{\mathcal{B}_0}\le  \sum_{j=0}^\infty 3^{-j}=\frac3{2}.$$
Since the norm on $\mathcal{B}_0$ is the supremum norm, it follows from the above  that the $\gl(n,\C)$-norm of $C_\Delta F(\la,z)$ is less than $\frac1{2}$. From \eqref{c-transform-eq} we conclude that $F(\lambda,z)\in
\GL(n,\C)$, and the proof is complete.
	\end{proof}

\begin{remark} \label{rem:Grauert} \rm (a) The approach used for the proof of the second part of the statement   is  similar to the one in \cite[Appendix, p.~665]{DPW}. In fact, using those ideas one can apply a theorem of Grauert \cite{Gr} to show that, if (iii) holds on a \emph{non-compact} Riemann surface $M$, then $V$ is a trivial holomorphic subbundle of $\HH$.

(b) If $V=\Phi\H_+$, where $\Phi$ is an extended  solution associated to  a $\U(n)$-valued harmonic map $\v$, the argument in the second part of Theorem \ref{d-bar-invariance} actually shows that $\Phi$ is real analytic on $M$. Indeed, any harmonic map $\v$ is real analytic on $M$ \cite{uhlenbeck}, hence, so is $\Phi^{-1}\partial_z\Phi$. Then one can replace the Banach algebras $\mathcal{B}_k$ in the proof of the second statement by  suitably chosen Banach algebras consisting of   functions which are holomorphic  in the second variable in the interior of $\Delta'\times \Delta'$ with $C^k$-extensions to the boundary. The argument produces a locally defined  $\GL(n,\C)$-valued function $w$, real analytic in $z$ and smooth in $\la$, such that $\Phi w$ is holomorphic in $z$,  hence $\Phi$ is real analytic in $z\in M$. As pointed out above, it follows from \cite[Thm 2.2]{uhlenbeck} that if $z_0\in M$ is fixed, then for every $z\in M$,  $\Phi^{-1}(\cdot,z_0)\Phi(\cdot,z)$ extends  to a holomorphic function in  $\C\setminus\{0\}$.

(c)  The converse of the first part is clearly true, that is, if
there exists a smooth $\U(n)$-valued function  $\psi$ on $S^1 \times M$ such that $V(z)=\psi(\cdot,z)\H_+$ then conditions (i) and (ii) hold.

(d) An extended solution $\psi$ satisfies the finite uniton number condition \eqref{finite-uniton-no} for some $r,s\in \N$ if and only if $W = \psi\H_+$ satisfies a similar equation:
$$
\la^s\H_+ \subseteq W \subseteq \la^{-r}\H_+.
$$
If $W$ satisfies this equation for some $r,s \in \N$, then it automatically satisfies conditions (i) and (ii) of the theorem.
\end{remark}

\subsection{The Gauss sequence} \label{subsec:the Gauss sequence}
Let $\varphi:M\to \U(n)$ be a harmonic map with associated smooth extended solution $\Phi$, and set $W(z)=\Phi(\cdot,z)\H_+$. We are interested in the structure of the nested sequence
 $$W_{(i)}:=W+\partial_z W+\cdots+\partial^i_zW.$$
The sequence is defined inductively, so we start our discussion with the first step $W_{(1)}$. A direct calculation shows that in terms of smooth sections we have $$W_{(1)}=\Phi(\la^{-1} \image{A_z^\varphi} +\H_+),$$
which immediately shows that $W_{(1)}(z)$ is a closed,  shift-invariant subspace of $L^2(S^1,\C^n)$ which also satisfies \eqref{propW}.
 However,  due to the fact that the  rank of $A_z^\v$ may vary,  $W_{(1)}(z)$ may not have the structure of a bundle, and the representation $W_{(1)}=\Phi_1\H_+$ with $\Phi_1$ smooth on $S^1\times M$ might fail. On the other hand,  as we shall see in
 \S \ref{subsec:further}, $A_z^\v$ is holomorphic with respect to a certain complex structure on a trivial bundle over $M$. In particular, the set of points  where its rank is not maximal  is discrete in $M$ and, furthermore, due to the technique of \emph{filling out zeros} developed in \cite{burstall-wood},  it follows that $\image{A_z^\varphi}$ has a unique extension  to a smooth subbundle $\alpha_1$ of the trivial bundle $M\times\C^n$. For this reason, we define
\begin{equation*} W_{(1)}=\Phi(\la^{-1} \alpha_1 +\H_+).
\end{equation*}
Note that, by continuity, $W_{(1)}$ satisfies \eqref{propW} at all points of $M$.

Now let $\varphi:M\to \U(n)$ be a harmonic map with associated smooth extended solution $\Phi$, and set $W(z)=\Phi(\cdot,z)\H_+$. Then $W_{(1)}$ is an extended solution which  obviously satisfies conditions (i) and (ii) of Theorem \ref{d-bar-invariance}. Indeed, from  $\la W_{(1)}\subset W$, we conclude that  $\bigcap_{k \in \N} S^kW_{(1)} = \{0\}$,
and
 by definition $$\dim W_{(1)}\ominus \la W_{(1)}=\dim\{\la^{-1}\alpha_1 +\alpha_1^\perp\}=n.$$
Thus, $W_{(1)}$ has   the form $\Phi_1\H_+$  and is a holomorphic subbundle of $\HH= M \times L^2(S^1,\C^n)$.  Therefore we
 can proceed to define inductively a nested  sequence  $(W_{(i)})$  of extended solutions by setting $W_{(0)}=W$, and $$W_{(i+1)}=(W_{(i)})_{(1)}.$$
 $W_{(i)}$ will be called the \emph{$i$th osculating subbundle of $W$}, and  $(W_{(i)})$  will be called the  \emph{Gauss sequence} of  $W$. Clearly, off a discrete set $D_i\subset M$ we have
 \begin{equation}
 \label{Gauss}
 W_{(i)}= \partial_zW_{(i-1)}+W_{(i-1)}=\partial^i_zW+\partial_z^{i-1}W+\cdots\partial_zW+W,
 \end{equation}
and each $W_{(i)}$  is an extended solution as well as a holomorphic subbundle of $\HH$.

 It is interesting to note that $W_{(i)} = \Phi_i\H_+$, where $\Phi_0=\Phi$
and
$\Phi_i$, $i \geq 1$, is the extended solution obtained from $\Phi_{i-1}$ by \emph{adding a uniton}:
\begin{equation}  \label{add-uniton}
 \Phi_i = \Phi_{i-1} (\pi_{\alpha_i} + \la^{-1}\pi_{\alpha_i}^{\perp}), \text{ equivalently, }
 \Phi_{i-1} = \Phi_i (\pi_{\alpha_i} + \la\pi_{\alpha_i}^{\perp});
 \end{equation}
in fact, $\alpha_i^{\perp} = \image A^{\varphi_{i-1}}_z$,
 where $\varphi_{i-1}$ is a harmonic map with associated extended solution $\Phi_{i-1}$, see, for example, \cite[Example 4.7]{svensson-wood-unitons}.

\subsection{The criterion} \label{subsec:the criterion}
In the sequel, we shall consider harmonic maps $\varphi:M\to \U(n)$   whose extended solutions  $W=\Phi\H_+$ are represented in the general form
\begin{equation}\label{W-psi}
W=\psi\H_+
\end{equation}
where $\psi:S^1 \times M \to \GL(n,\C)$ is a smooth map.
To do calculations cleanly, we will choose a local complex coordinate $z$; but as previously mentioned, our results will not depend on that choice.

Recall from \S \ref{sec:intro} that
$\psi^{-1}(\la,z)\partial_z\psi(\la,z)$ extends to a meromorphic function of $\la$ in the punctured unit disc, with a simple pole at the origin, i.e.,
\begin{equation} \label{deriv-psi}
A_{\psi}(\la,z): = \psi^{-1}(\la,z)\partial_z\psi(\la,z)
=\sum_{j=-1}^\infty a_j(z)\lambda^j,\quad \la\in S^1,\end{equation}
where the $a_j$ are smooth
matrix-valued functions on $M$, and the series together with its derivatives converges uniformly on compact subsets of   $S^1\times M$.

The construction of the Gauss sequence of $W$ given in the previous subsection  is related to the action of a family of differential operators induced by $\psi$ on $C^\infty(M,\C^n)$, which will play a crucial role in what follows (see also Remark \ref{dmodules} for an interpretation in terms of $D$-modules).

More precisely, for   $\la\in S^1$, let   $T = T_\psi(\la):C^\infty(M,\C^n)\to C^\infty(M,\C^n)$ be defined by
\begin{equation}\label{def-T}
T_\psi(\la)u=\partial_zu+ \psi^{-1}\partial_z\psi(\la,\cdot)u.\end{equation}
Given any nonnegative integer power $r$  of this operator and $u\in  C^\infty(M,\C^n)$, we shall regard throughout  $T^r_\psi(\la)u(z)$ as a smooth function on $S^1\times M$. It follows from  \eqref{deriv-psi} that
$T^r_\psi(\cdot)u(z)\in \la^{-r}\H_+,~z\in M$.

\begin{lemma}\label{T_psi}  Let $W=\psi\H_+$ be an extended solution with $\psi^{-1}\partial_z\psi$ given by
\eqref{deriv-psi}.	

Then off a discrete set $D_i\subset M$ we have
 \begin{equation}\label{W_i-form}
 W_{(i)}=\psi(\lambda,z)\big\{\sum_{r=1}^iT^r_\psi(\la)h_r +v(\cdot,z):~h_r\in C^\infty(M,\C^n),\, v(\cdot,z)\in \H_+  \big\}.	\end{equation}
\end{lemma}
\begin{proof} This follows easily by induction.
	For $i=0$,  \eqref{W_i-form} is obvious and for $i\ge 1$, by \eqref{Gauss} there exists $D_i\subset M$ discrete such that
	$$W_{(i)}= \partial_z W_{(i-1)}+ W_{(i-1)}$$
off $D_i$.	
\end{proof}

Our criterion for finiteness of the uniton number can be formulated as follows.

\begin{theorem}\label{finite uniton} Let $W=\psi\H_+$ be an extended solution with $\psi^{-1}\partial_z\psi$ given by
\eqref{deriv-psi}.	
	Then the following are equivalent$:$
\begin{enumerate}	
	\item[(i)] $W$ is of finite uniton number$;$
	\item[(ii)] there exists a smooth map $\hat{\Phi}:S^1\to \U(n)$ (independent of $z\in M$) such that $W\subset \hat{\Phi}\H_+;$
	\item[(iii)] the Gauss sequence of $W$  stabilizes, i.e., there exists $i\ge 0$ such that $W_{(i+1)}=W_{(i)};$
	\item[(iv)] if $T=T_\psi(\la),~0<|\la|<1$, denotes the operator defined in \eqref{def-T} then the maximum power of\/ $\la^{-1}$ in  $T^r$ stays bounded when $r\in \mathbb{N}$, i.e., there exists $k_0\in\mathbb{N}$ such that $$\int_0^{2\pi}e^{ikt}T_\psi^r(e^{it})udt=0, \quad k\ge k_0,\,u\in C^\infty(M,\C^n),\, r\in \mathbb{N}.$$
\end{enumerate}	
\end{theorem}
\begin{proof} (i) $\Rightarrow$ (ii) is obvious by definition, since if (i) holds, all extended solutions have the form $v\Phi_0$, with $v$ constant in $z\in M$, and $\Phi_0$ a trigonometric polynomial in $\lambda$.

(ii) $\Leftrightarrow$ (iii). If (ii) holds, on writing the bundles $(W_{(i)})_{i\ge 0}$ as $$W_{(i)}=\Phi_{i}(\cdot, z)\H_+$$ with the $\Phi_i$ extended solutions, then $\hat{\Phi}^{-1}\Phi_{i}(\cdot,z),~i\ge 0$, are the boundary values of bounded analytic matrix-valued functions in the unit disc, and by \eqref{add-uniton} there exist projection-valued smooth functions $\pi_i,~i\ge 1$  on $M$ such that
	$$\hat{\Phi}^{-1}(\lambda)\Phi_{0}(\lambda,z)=
\hat{\Phi}^{-1}(\lambda)\Phi_{i}(\lambda,z)\prod_{l=1}^i \bigl(\pi_l(z)+\lambda\pi_l^\perp(z) \bigr)$$
	for all $i\ge 1$ and all $\lambda$ in the unit disc. This implies that the order of the zero at the origin  of
	$\det \prod_{l=1}^i \bigl(\pi_l(z)+\lambda\pi_l^\perp(z)\bigr)$ cannot exceed the order of the zero at the origin of  $\det\hat{\Phi}^{-1}(\lambda)\Phi_0(\lambda,z)$, for all $i\ge 1$, hence there is an $i_0 \in \N$ such that $\pi_i(z)=I$ for $i\ge i_0$, and thus $W_{(i)}=W_{(i_0)},~i\ge i_0$. Conversely, if  $\varphi_i=\Phi_{i}(-1,\cdot)$, it follows from the formula for the uniton \eqref{add-uniton} that
	$W_{(i+1)}=W_{(i)}$ is equivalent to $A_z^{\varphi_i}=0$, and this implies that $\Phi_{i}$ is constant in $z\in M$, i.e., (ii) holds with $\hat{\Phi}=\Phi_{i}$.
	
	(iii) $\Leftrightarrow$ (iv)   follows  immediately from Lemma \ref{T_psi}. Finally if (iii) holds, then as above, we deduce that
	the extended solution $\Phi_i$ corresponding to $W_{(i)}$ is constant and  (i) follows.
\end{proof}

\begin{remark} \label{rem:extd local} {\rm
By real analyticity, our criteria are local, i.e., if any of the above criteria holds on an open set, it holds on the whole of} $M$.
\end{remark}

\subsection{Direct applications} \label{subsec:direct applications}

Our first observation extends a result in \cite[Example 4.2]{svensson-wood-twistor-lifts} to the case of arbitrary potentials.
\begin{corollary}\label{a-1 nilpotent} Let\/ $W=\psi\H_+$ be an extended solution with $\psi^{-1}\partial_z\psi$ given by
	\eqref{deriv-psi}.
	 If\/ $W$ is of finite uniton number then $a_{-1}$ is nilpotent.	
	\end{corollary}
	
	\begin{proof}
	This is just an application of Theorem \ref{finite uniton}(iv) since 	
$$\int_0^{2\pi}e^{irt}T^r(e^{it})udt=a_{-1}^ru, \quad u\in C^\infty(M,\C^n),\, r\in \mathbb{N}.$$				
		\end{proof}

Note that, if $\psi$ is an extended solution $a_{-1} = A^{\v}_z$.		
The condition is far from sufficient, as follows from the results and examples below.\\
In  general, the powers of the differential operators $T_\psi(\la)$ are difficult to control, but there are at least two extreme situations when the problem becomes more tractable. Throughout in what follows we shall use the fact that the $W_{(i)}$ are holomorphic subbundles of $\HH= M \times L^2(S^1,\C^n)$.

\begin{proposition} \label{extreme cases}  Let $W=\psi\H_+$ be an extended solution with $\psi^{-1}\partial_z\psi$ given by
	\eqref{deriv-psi}.
\begin{enumerate}	
\item[(i)]  Assume that  $a_{-1}^2=0$, and  $$\image{\bigl(a_0(z)a_{-1}(z)+\partial_za_{-1}(z)\bigr)}+\image{a_{-1}(z)}=\C^n$$
 for all $z$ on an open non-void  subset $U \subset M$.
Then for $i\ge 0$,  $W_{(2i)}= \la^{-i}\psi\H_+  = \la^{-i}W$.  In particular,  $W$ is not of finite uniton number.
\item[(ii)] Assume that
$$\image{\bigl(a_j(z)a_{-1}(z)+\partial_za_{-1}(z)\bigr)}\subset\image{a_{-1}(z)},\quad j\ge -1,$$
 for all $z$ on an open non-void subset $U \subset M$.
 Then for all $i\ge 0$,
$$W_{(i)}(z)\subset \psi(\cdot,z)\bigl(\sum_{r=1}^i\image{a_{-1}^r(z)} +\H_+ \bigr).$$
In this case, $W$ is of finite uniton number if and only if $a_{-1}$ is nilpotent.
\end{enumerate}
\end{proposition}

\begin{proof} (i) Note that the second part of the statement follows from the first part together with Theorem \ref{finite uniton}(iii).
	By Lemma \ref{T_psi} we have
$$W_{(2)}(z)=\psi(\cdot,z)\{T^2_\psi(\cdot)h_2(z)+T_\psi(\cdot)h_1(z)+v:~h_1,h_2\in C^\infty(M,\C^n),\, v\in \H_+\},$$
on $U\setminus D_2$, where $D_2\subset M$ is discrete.
 Using the notation in  \eqref{deriv-psi}, we obtain after a direct computation,
\begin{align*}&T^2_\psi(\la)h_2(z)+T_\psi(\la)h_1(z)=\\
&\la^{-1}\bigl(a_0(z)a_{-1}(z)h_2(z)+\partial_za_{-1}(z)\bigr)h_2(z)
+a_{-1}(z)\bigl(a_0(z)h_2(z)+2\partial_zh_2(z)+h_1(z)\bigr)\\
& \space\space\space\space +\la^{-2}a_{-1}^2h_2(z)+v_1(\la,z),
\end{align*}
with $v_1(\cdot,z)\in \H_+$. Thus, by assumption,
\begin{align*}W_{(2)}=&\psi\bigl\{\la^{-1}\bigl((a_0a_{-1}+\partial_za_{-1})h_2+a_{-1}h_3 \bigr) \
:~ h_2,h_3\in C^\infty(M,\C^n)\bigr\} + \psi\H_+ \\
=&\la^{-1}\psi \H_+ = \la^{-1}W, \end{align*}
on $U\setminus D_2$.
 This immediately implies  that the first part of the statement  holds for each $i\ge 0$ on $U\setminus D_i$ with $D_i$ discrete, and so by analytic continuation, on the whole of  $M$.
 \smallskip

(ii)  Again working on $U$, set $$V_i(z)=\psi(\la,z)\bigl(\bigl\{\sum_{r=1}^i\la^{-r}a_{-1}^r(z)h_r(z):~h_r\in C^\infty(M,\C^n),\,1\le r\le i\bigr\} +\H_+ \bigr).$$
Using the identity $$\partial_z  a_{-1}^r=(\partial_z a_{-1})a_{-1}^{r-1}+\sum_{k+l=r-1}a_{-1}^k(\partial_za_{-1})a_{-1}^l, \quad r\ge 1,$$
together with the assumption, it follows easily that
$$\partial V_i+V_{i-1}\subset V_{i+1}.$$
Moreover, $V_1=W_{(1)}$ on $U\setminus D_1$, hence we derive by induction that $W_{(i)}\subset V_i$ on $U\setminus D_i$, for all  $i\ge 1$, where the $D_i\subset M$ are discrete.
If $a_{-1}$ is nilpotent, we see that the Gauss sequence of $W$ stabilizes on some open set $U'$, and so, by analytic continuation, on $M$, hence $W$ is of finite uniton number by Theorem \ref{finite uniton}(iii). The converse is given by Corollary \ref{a-1 nilpotent}.
	\end{proof}

 By allowing the occurrence of singularities, it was shown in \cite{DPW} that any extended solution, and so harmonic map, arises from a solution to   $\psi^{-1}\partial_z\psi=\lambda^{-1}a_{-1}(z)$, with $a_{-1}$ a \emph{meromorphic} function on $M$; such a $\lambda^{-1}a_{-1}(z)$ is called a \emph{meromorphic potential}.

\begin{corollary} \label{cor:extreme cases}  Let $W=\psi\H_+$ be an extended solution with meromorphic potential $\psi^{-1}\partial_z\psi=\lambda^{-1}a_{-1}(z)$ with
$a_{-1}$ not identically zero.
\begin{enumerate}
\item[(i)]  Assume that  $a_{-1}^2=0$, and
\begin{equation} \label{dim-assumption}
 \image{\partial_za_{-1}(z)}+\image{a_{-1}(z)}=\C^n
 \end{equation}
 for all $z$ on an open non-void subset $U \subset M$.
Then $n$ is even and $W$ is not of finite uniton number.
\item[(ii)]  Assume that
$$\image{\partial_za_{-1}(z)}\subset\image{a_{-1}(z)}$$
 for all $z$ on an open non-void subset $U \subset M$.
 Then  $W$ is of finite uniton number if and only if $a_{-1}$ is nilpotent.
 \end{enumerate}
\end{corollary}

\begin{proof} (i)  We work on the open subset $V$ of $U$ on which $a_{-1}$ has maximal rank; call this rank $m$.
{}From $a_{-1}^2=0$, we have $\image{a_{-1}} \subseteq \ker a_{-1}$.  But, as always, $\dim\image{a_{-1}} + \dim\ker{a_{-1}} = n$ so that $2m \leq n$.  On the other hand, the left-hand side of \eqref{dim-assumption} has dimension at most $2m$ so that $n \leq 2m$.
The rest is immediate from part (i) of the proposition.

(ii)  This is a direct consequence of Proposition  \ref{extreme cases}(ii).
\end{proof}

\section{Applications to harmonic maps}

\subsection{Applications of Theorem \ref{finite uniton}} \label{subsec:applns ha maps}
We now apply our criterion for finiteness of the uniton number to harmonic maps from a Riemann surface $M$.  Although these have extended solutions locally, they may not have extended solutions defined on the whole of $M$; this means that Remark \ref{rem:extd local} does not immediately apply to harmonic maps.  However, we have the following result, which we were unable to find elsewhere; it is proved using our criterion.

\begin{proposition} \label{finite uniton descends}  Let $M$ and $\wt{M}$ be Riemann surfaces and $\Pi:\wt{M} \to M$ a non-constant holomorphic map. Let $\v:M \to \U(n)$ be a harmonic map.  Then
$\v$ is of finite uniton number if and only if\/ $\v \circ \Pi$ is of finite uniton number.
\end{proposition}
\begin{proof} `Only if' follows by noting that if $\Phi$ is a polynomial extended solution associated to $\v$, then $\Phi \circ \Pi$ is a polynomial extended solution associated to $\v \circ \Pi$.

Conversely, suppose that $\wt\v_0:= \v \circ \Pi$ is of finite uniton number.  Let $\wt\Phi_0$ be an extended solution on $\wt{M}$ with
$\wt\Phi_0(-1,\cdot) = \wt\v_0$.
Then, from \eqref{add-uniton}, its Gauss sequence $\wt{W}_{(i)} = \wt{\Phi_i}\H_+$ on $\wt{M}$ is given inductively by
\begin{equation*}
\wt{\Phi}_i
=\wt{\Phi}_{i-1}(\pi_{\wt\alpha_i} + \la^{-1}\pi_{\wt\alpha_i}^{\perp}), \quad
\text{and so} \quad
\wt\v_i =\wt\v_{i-1}(\pi_{\wt\alpha_i} - \pi_{\wt\alpha_i}^{\perp}),
\end{equation*}
where $\wt\v_i = \wt\Phi_i(-1,\cdot)$ and $\wt\alpha_i^{\perp} = \image A^{\wt\v_{i-1}}_{\wt z}$ for any local complex coordinate $\wt z$.
Since $\wt\v_0$ is of finite uniton number, this terminates, i.e. there is an $r \in \N$ and $\ga:S^1 \to \U(n)$ such that
$\wt\Phi_r(\cdot,z) = \ga$ for all $z \in M$, equivalently,
$\wt\v_r$ is constant; so
$$\wt\Phi_0 = \ga (\pi_{\wt\alpha_r} + \la\pi_{\wt\alpha_r}^{\perp}) \cdots
(\pi_{\wt\alpha_1} + \la\pi_{\wt\alpha_1}^{\perp}).$$

Now, from $\v:M \to \U(n)$ we can define a sequence of maps $\v_i:M \to \U(n)$ by $\v_0 = \v$ and
$\v_i =\v_{i-1}(\pi_{\alpha_i} - \pi_{\alpha_i}^{\perp})$,
where $\alpha_i^{\perp} = \image A^{\v_{i-1}}_z$ for any local complex coordinate $z$.
We claim that, for $i=1,2,\ldots, r$, (i) $\wt\v_{i-1} = \v_{i-1} \circ \Pi$ and (ii) $\wt\alpha_i = \alpha_i \circ \Pi$.
By construction, (i) holds for $i=1$. Assume that (i) holds for some $i$, then,
away from the branch points of the holomorphic map $\Pi$, working with local complex coordinates $\wt z$ and $z=\Pi(\wt z)$ of $\wt M$ and $M$, respectively, and filling out zeros at branch points, we have
$\wt\alpha_i^{\perp} = \image A^{\wt\v_{i-1}}_{\wt z} = \image A^{\v_{i-1}}_z \circ \Pi = \alpha_i^{\perp}$ so that,  from $\v_i =\v_{i-1}(\pi_{\alpha_i} - \pi_{\alpha_i}^{\perp})$ we have $\wt\v_i = \v_i \circ \Pi$, completing the induction step.

In particular, since $\wt\v_r$ is constant, so is $\v_r$ on the open set $\Pi(\wt M)$ and so, by analytic continuation, on the whole of $M$.  Hence, $\v$ is of finite uniton number.  In fact, $\wt\Phi_0 = \Phi_0 \circ \Pi$, where
$$
\Phi_0 = \ga (\pi_{\alpha_r} + \la\pi_{\alpha_r}^{\perp}) \cdots
(\pi_{\alpha_1} + \la\pi_{\alpha_1}^{\perp});
$$
$\Phi_0$ is an extended solution of $\v$ on the open set $\Pi(\wt M)$ and so, by analytic continuation, on the whole of $M$.
\end{proof}

Recall that any Riemann surface $M$ has a simply connected covering space $\wt{M}$ with holomorphic projection
$\Pi: \wt{M} \to M$.  Since any harmonic map on a simply connected Riemann surface has an extended solution, we obtain the following non-trivial extension to Remark \ref{rem:extd local}.

{\begin{corollary} \label{cor:fi uniton no local}
Let $\v:M \to \U(n)$ be a harmonic map, and let $\Pi: \wt{M} \to M$ be a simply connected covering space. Then the following are equivalent$:$
\begin{enumerate}
  \item[(i)]  $\v$ is of finite uniton number$;$
   \item[(ii)]  $\v|_U$ is of finite uniton number  for some open set $U$ of $M$$;$
   \item[(iii)] $\v \circ \Pi:\wt{M} \to \U(n)$ is of finite uniton number.
\end{enumerate}
\end{corollary}
\begin{proof}
  The equivalences follow by applying Proposition \ref{finite uniton descends} to the covering map  $\Pi:\wt{M} \to M$ and the inclusion map of $U$ in $M$.
\end{proof}

In order to apply the criteria of Theorem \ref{finite uniton}, we can choose $\psi$ to be a local extended solution of $\v$.  Note that any two extended solutions differ by premultiplication by a constant loop, which does not affect
any of the criteria. In particular, $A_\psi(\lambda,\cdot)=(1-\la^{-1})A_z^\v$ so that
$T_{\psi} = T_{\v}$ where we define $T_{\v} = \pa_z + (1-\la^{-1})A^{\v}_z$.  Using Corollary \ref{cor:fi uniton no local}, we deduce the following from Theorem \ref{finite uniton}, which shows that we can use $T_{\v}$ to test a harmonic map for finite uniton number \emph{without assuming that it has an extended solution on the whole of its domain}.

\begin{corollary} \label{cor:T criterion for ha map}
Let $\v:M \to \U(n)$ be a harmonic map. Then $\v$ is of finite uniton number if and only if the maximum power of\/ $\la^{-1}$ in  $T_{\v}^r$ stays bounded on some open subset of $M$ when $r\in \mathbb{N}$.
\qed\end{corollary}

The sufficient condition for the uniton number to be finite provided by Proposition \ref{extreme cases}(ii) is then useful as the potential $A_\psi$ is a polynomial of degree one in $\lambda^{-1}$:

\begin{corollary}\label{extreme for extended} Let $\v:M \to \U(n)$ be a harmonic map.
	 If\/ $\image{\partial_zA_z^\v}\subset\image{A_z^\v}$
 for all $z$ on an open non-void subset $U \subset M$, then $\v$ has finite uniton number if and only if  $A_z^\v$ is nilpotent.
\end{corollary}

\begin{proof} On some open subset of $U$, $\v$ has an extended solution $\psi$. Then, as above, $A_\psi(\lambda,\cdot)=-\la^{-1}A_z^\v+A_z^\v$, and the result is a direct application of Proposition \ref{extreme cases}(ii)
and Corollary \ref{cor:fi uniton no local}.
	\end{proof}
	
In the case when $n=2$,  Proposition \ref{extreme cases}  leads to a dichotomy and, in particular,  provides a complete characterization of harmonic maps with finite uniton number:

\begin{corollary}\label{U(2)} Let $\v:M\to \U(2)$ be  harmonic and non-constant.
Then, \emph{either}
\begin{enumerate}	
\item[(i)]  at points where $A_z^\v$ is non-zero, $\image{\partial_zA_z^\v}\subset\image{A_z^\v}$.  Then $\v$ has finite uniton number if and only if $(A_z^\v)^2=0;$ \emph{or}
\item[(ii)] $\image{\partial_zA_z^\v}+\image{A_z^\v}=\C^2$ on an open subset of $M$. Then
$\v$ does not have finite uniton number; indeed,
for any extended solution on any open subset of $M$, the corresponding  Gauss sequence satisfies  $W_{(2i)}=\la^{-i}W$.
\end{enumerate}
\end{corollary}

\begin{proof}
Clearly, $A_z^\v$ is nilpotent if and only if $(A_z^\v)^2 = 0$.
	The set where $\image{\partial_zA_z^\v}+\image{A_z^\v}=\C^2$
is open in $M$. If this set is void, we have $\image{\partial_zA_z^\v}\subset\image{A_z^\v}$ at points where $A_z^\v$ is non-zero and (i) follows from the previous corollary.  Otherwise, we obtain (ii) by part (i) of Proposition \ref{extreme cases}.
		\end{proof}

The condition for finiteness of the uniton number is easy to understand in this particular case.
 If $A_z^\v$ is nilpotent, then it has the form $$A_z^\v= \begin{pmatrix} a & b \\ c & -a  \end{pmatrix},$$
with $a^2+bc=0$. A straightforward computation
reveals  that  $\image{\partial_zA_z^\varphi}\subset \image{A_z^\v}$  holds if and only if $\partial_za=0$ whenever $a=0$, and on the open set where $a\ne 0$, both functions
$b/a$, $c/a = -a/b$ are antiholomorphic.

To interpret this, let $g:M \to \CP^1$ be the map into complex projective $1$-space given by the image of $A^{\varphi}_z$, thus $g(z) = [a(z),c(z)] = [b(z),-a(z)]$ in homogeneous coordinates; note that $g$ is antiholomorphic. Let $f:M \to \CP^1$ be the holomorphic map orthogonal to it, i.e.,
$f(z) = [-\ov{c(z)}, \ov{a(z)}] = [\ov{a(z)}, \ov{b(z)}]$.  Then $A^f_z$ has image and kernel given by $g$, i.e., the same image and kernel as $A^{\varphi}_z$. In fact, $A^{\varphi}_z = A^f_z$. To see this, set $s = \bigl(1, \ov{b(z)}/\ov{a(z)}\bigr)$, a section of $f$. By matrix multiplication, $A^{\varphi}_z(s) = (a+b\ov{b}/\ov{a}, c-a\ov{b}/\ov{a})$.  On the other hand, adding equations \eqref{harmcond} and \eqref{integrability} gives $\pa_{\bar{z}}A^{\v}_z = [A^{\v}_z, A^{\v}_{\zbar}]$.  Now $A^{\v}_{\zbar}$ is minus the adjoint of $A^{\v}_z$, thus
$A_{\zbar}^\v= \begin{pmatrix} -\ov{a} & -\ov{c} \\ -\ov{b} & \ov{a} \end{pmatrix}$ so that the last equation gives
$\pa_z \ov{a} = c\ov{c} - b\ov{b}$ and $\pa_z \ov{b} = 2(a\ov{b} - \ov{a}c)$.  Using these equations and $a^2 +bc=0$, it can be checked that $A^f_z(s)
 := \pi_g(\pa_z s) = A^{\varphi}_z(s)$.  It follows that $A^f_z = A^{\varphi}_z$ so that, up to premultiplication by a constant unitary matrix, $\v$ is the holomorphic map $f$. That $\varphi$ is holomorphic also follows from Uhlenbeck's work \cite{uhlenbeck}.

\subsection{Constant potentials} \label{subsec:constant potentials}
 Here we consider an interesting  particular case of the  Dorfmeister--Pedit--Wu construction \cite{DPW}. We shall be concerned with extended solutions
$W(z)=\psi(\cdot, z)\H_+$  where $\psi:S^1\times M\to \GL(n,\C)$ is a smooth map as in \eqref{W-psi}, but now we assume that the potential   $A_\psi=\psi^{-1}(\la,z)\partial_z\psi(\la,z)$ is constant in $z\in M$, i.e., there exists a smooth $\gl(n,\C)$-valued map $p$  on $S^1$, constant in $z\in M$, such that the negative Fourier coefficients of $\la p(\la)$  vanish, and
\begin{equation}\label{const-potential}
A_{\psi}(\la,z) =p(\la), \quad \la \in S^1;\end{equation}
here we take $M$ to be an open subset of $\C$ so that we have a global coordinate $z$.
 The simplest examples are the so-called \emph{vacuum solutions}, see Example \ref{ex:Clifford}; they are given by setting $\tau(z)$ equal to
$z$, or a linear polynomial in $z$, in the formula
\begin{eqnarray}
\Phi(\la,z) &=& \exp\bigl(2i\,{\mathcal Im}\,\{\tau(z)(1-\la^{-1})A\}\bigr) \label{vacuum} \\
		&=& \exp\bigl(\tau(z)(1-\la^{-1})A-\overline{\tau(z)}(1-\la)A^*\bigr),\nonumber
\end{eqnarray}
where $A$ is a non-zero constant matrix which is normal ($AA^*=A^*A$).  More generally,
any non-compact Riemann surface $M$ admits non-constant holomorphic functions $\tau:M\to\C$, see,  for example,
\cite[Ch.~3]{forster}, and so we obtain vacuum solutions on $M$.
 Such extended solutions are not of finite uniton number by Corollary \ref{a-1 nilpotent}, since a non-zero normal matrix cannot be nilpotent. In particular, this shows that \emph{on any non-compact Riemann surface $M$ there are harmonic maps into $\U(n)$ which are not of finite uniton number}.

We now extend to $\U(n)$ a result proved for $\SU(2)$ in \cite[Proposition A.4]{burstall-pedit}.

\begin{theorem} \label{constant potentials}  Let $W=\psi\H_+$ be an extended solution with $A_{\psi} = \psi^{-1}\partial_z\psi$ given by \eqref{const-potential}. Then $W$ has finite uniton number if and only if $\det(\mu I-\psi^{-1}\partial_z\psi)$ is holomorphic in $\{(\la,\mu)\in \C^2:~|\la|<1\}$.
\end{theorem}

\begin{proof}
If $p=\psi^{-1}\partial_z\psi$ is constant in $z\in M$, then $$T^r_\psi(\la)h=\sum_{j=0}^r{r\choose j}p^j\partial_z^{r-j}h,\quad h\in C^\infty(M,\C^n);$$
hence by  Lemma \ref{T_psi} it follows immediately that
\begin{equation}\label{w-const-pot} W_{(i)}(z)=\psi(\cdot, z)\bigl\{\sum_{j=0}^ip^jh_j(z)+v:~h_j\in  C^\infty(M,\C^n), 0\le j\le i, v\in \H_+ \bigr\}.\end{equation}
If $W$ has finite uniton number it follows from Theorem \ref{finite uniton}(iv) that there exists a nonnegative integer $m$ such that
given any polynomial $Q$, the function $Q(p)$ is holomorphic in $\{|\la|<1\}\setminus \{0\}$ with a pole of order at most $m$ at the
 origin.
 In particular, for $\mu\in \C$ and all positive integers $k$,  the function $\det \bigl((\mu I-p)^k \bigr)=\det^k(\mu I-p)$    is holomorphic in $\{|\la|<1\}\setminus \{0\}$ with a pole of order at most $mn$ at the origin. This clearly implies that $\det(\mu I-p)$ actually extends analytically at the origin.\\
Conversely, assume that  $\det(\mu I-p)$ is  holomorphic in $\{|\la|<1\}\times \C$.  Note that, for fixed $\la\in S^1$, the polynomial $Q_\la(\mu)=\det(\mu I-p(\la))$ is the characteristic polynomial of $p(\la)$ and has dominant coefficient equal to $1$. Therefore, it follows easily by induction that for $k\ge n$ we have $p^k(\la)=R_k(\la,p(\la))$, where $R_k(\la,\cdot)$ is a polynomial of degree strictly less than $n$  whose coefficients extend analytically  to the unit disc.
 In fact, $$R_n(\la,\mu)=\mu^n-Q_\la(\mu), \quad R_{k+1}(\la, \mu)=\mu R_k(\la, \mu)-Q_\la(\mu).$$
Then for $i>n$, and $h_j\in  C^\infty(M,\C^n),~ 0\le j\le i$,
$$\sum_{j=0}^ip^j(\la)h_j(z)=\sum_{j=0}^{n-1}p^j(\la)h_j(z)+\sum_{j=n}^{i}R_j(\la,p(\la))h_j(z); $$
hence  the maximum power of $\la^{-1}$ in these sums stays bounded when $i>n$. Then by  \eqref{w-const-pot} and Theorem \ref{finite uniton}(iv), $W$ is of finite uniton number.
\end{proof}

\subsection{Further applications and examples} \label{subsec:further}
In this section we shall interpret the bounded powers criterion, Theorem \ref{finite uniton}, and give some applications to harmonic maps into Grassmannians. Our treatment will be more geometrical than previous sections, and will use some constructions of harmonic maps found in the literature.  Let $W=\Phi\H_+$, with  $\Phi$ an extended solution associated to some harmonic map $\varphi:M \to \U(n)$ from a Riemann surface.
We think of $W$ as a subbundle of the trivial bundle $\HH = M \times L^2(S^1,\C^n)$ whose fibre at $z \in M$ is $W(z) = \Phi(z)\H_+$.  We shall also need the trivial bundle $\HH_+ := M \times \H_+$; this is a subbundle of $\HH$.
Setting $\psi = \Phi$, as mentioned in \S \ref{subsec:applns ha maps}
the differential operator $T=T_{\v} = T_{\psi}$ defined by \eqref{def-T} can be written:
\begin{equation} \label{T}
T =\pa + (1-\lambda^{-1})A^{\varphi}_z = -\lambda^{-1}A^{\varphi}_z + D^{\varphi}_z\,.
\end{equation}
Here $A^{\varphi}_z$ is defined by \eqref{Aphi}
and $D^{\varphi}_z$ is the derivation given by $D^{\varphi}_z = \pa_z + A^{\varphi}_z$;  similarly, we
set $D^{\varphi}_{\zbar} = \pa_{\zbar} + A^{\varphi}_{\zbar}$.  The $1$-form $D^{\varphi}_z dz + D^{\varphi}_{\zbar}d\zbar$ may be interpreted as a \emph{unitary connection} giving a \emph{covariant derivative} on the trivial bundle $\CC^n:= M \times \C^n$. As in the proof of Corollary \ref{U(2)},  adding the harmonic equation \eqref{harmcond} and the integrability equation \eqref{integrability} gives $\pa_{\zbar}A^{\v}_z = [A^{\v}_z, A^{\v}_{\zbar}]$.
This  can be written as
$D^{\varphi}_{\zbar}A^{\varphi}_z = A^{\varphi}_zD^{\varphi}_{\zbar}$ \cite[\S 1]{uhlenbeck}, which can be interpreted as saying that
\emph{$A^{\varphi}_z$ is a holomorphic endomorphism of\/ $\CC^n$ with respect to the holomorphic structure on $\CC^n$ with $\dbar$-operator given by $D^{\varphi}_{\zbar}$}, see, for example,
\cite[\S 3.1]{wood-60}.
We extend these operators to (local) sections $\sigma$ of $\HH$\,, thought of as mappings $S^1 \times M \to \C^n$, by applying them to $\sigma(\la,\cdot)$ for each $\la$, for example,
$\bigl(D^{\varphi}_z(\sigma)\bigr)(\la, \cdot) = D^{\varphi}_z\bigl(\sigma(\la,\cdot)\bigr)$.

By \cite[Proposition 3.9]{svensson-wood-unitons}, there is a commutative diagram:
\begin{equation}
\begin{gathered}\label{diag:F-Az-comm}
\xymatrixrowsep{1.5pc}\xymatrixcolsep{1.3pc}
\xymatrix{
	\Gamma(W)  \ar[d]_{\Phi^{-1}}  \ar[r]^(0.48){\pa_z}  &  \Gamma(W_{(1)})\ar[d]_{\Phi^{-1}} \ar[r]^{\pa_z} &  \Gamma(W_{(2)})\ar[d]_{\Phi^{-1}}   \ar[r]^(0.65){\pa_z} &  \cdots \\
	\Gamma(\HH_+)                 \ar[r]_(0.4){T}   &  \Gamma(\la^{-1}\HH_+)               \ar[r]_{T}   &  \Gamma(\la^{-2}\HH_+)     \ar[r]_(0.7)T & \cdots
}
\end{gathered}
\end{equation}
thus $\Phi^{-1}$ intertwines $\pa_z$ and $T$.  Similarly,
$\Phi^{-1}$ intertwines $\pa_{\zbar}$ and $\ov{T} := \pa_{\zbar} + (1-\lambda)A^{\varphi}_{\zbar}$ so that
$T$ is holomorphic in the sense that it commutes with $\ov{T}$ --- this can also be checked directly from the harmonic equation.

 \begin{remark}\label{dmodules}
{\rm  The commutativity of diagram \eqref{diag:F-Az-comm} reflects the fact that $\Gamma(W)$ and $\Gamma(\HH_+)$ are isomorphic as \emph{$D$-modules} (see \cite[\S 8.2]{guest-book2}).
 In view of  the harmonicity conditions \eqref{propW}, the appropriate  structure of $D$-module to be considered on $\Gamma(W)$ is that induced by the differential operators $\lambda \pa_z$ and  $ \pa_{\bar z}$ (see \cite[Example 8.5]{guest-book2}). The \emph{gauge transformation} $\Phi^{-1}$ converts these differential operators acting on
 $\Gamma(W)$ in the operators $\lambda T$ and $\overline{T}$, respectively, acting on   $\Gamma(\HH_+)$, exhibiting then  $\Gamma(W)$ and $\Gamma(\HH_+)$ as isomorphic $D$-modules.}
  \end{remark}

In \eqref{diag:F-Az-comm}, $W_{(i)}$ denotes the $i$th osculating subbundle of $W$ defined inductively by \eqref{Gauss}.  Note that, although the rank may drop at isolated points,
$W_{(i)}$ can be extended to a subbundle of $\HH$ as explained in \S \ref{subsec:the Gauss sequence}.
In similar fashion to \eqref{Gauss},  define $\Ss_i = \Ss_i(T)$ inductively by $\Ss_0 = \HH_+$, $\Ss_i = T(\Ss_{i-1}) + \Ss_{i-1}$, \ $i = 1,2,\ldots$.
By the commutativity of the diagram \eqref{diag:F-Az-comm}, $\Ss_i = \Phi^{-1}W_{(i)}$ so that $W_{(i)} = \Phi \Ss_i$, cf.\ Lemma \ref{T_psi}.
Then, if we extend $W_{(i)}$ to a subbundle, we can correspondingly extend $\Ss_i$ to the subbundle  $\Phi^{-1}W_{(i)}$, so that $W_{(i)} = \Phi \Ss_i$ \emph{as smooth subbundles}.

Hence, Corollary \ref{cor:T criterion for ha map} becomes the following:
\begin{corollary} \label{S-stabilizes} Let $\v:M \to \U(n)$ be a harmonic map, Then the following are equivalent$:$
\begin{enumerate}
\item[(i)]  $\v$ has finite uniton number$;$
\item[(ii)] $\Ss_i(T)$ stabilizes, i.e., there exists $i \in \N$ such that $\Ss_i(T) = \Ss_{i+1}(T);$
\item[(iii)] the maximum power of $\la^{-1}$ occurring in $\Ss_i(T)$ is bounded for $i \in \N$.
\end{enumerate}
\end{corollary}

We now discuss \emph{harmonic maps into Grassmannians} and geometrical methods for studying them.
For $k \in \{0,1,\ldots, n\}$, let $G_k(\C^n)$ denote the Grassmannian of $k$-dimensional subspaces of $\C^n$. This can be embedded in $\U(n)$ by the
\emph{Cartan embedding} \cite[Proposition 3.42]{cheeger-ebin}, which is given, up to left-multiplication by a constant, by $\alpha \mapsto \pi_{\alpha} - \pi_{\alpha}^{\perp}$.  Since the Cartan embedding is totally geodesic,
by the composition law \cite[\S 5]{eells-sampson}
it preserves harmonicity, i.e.,
\emph{a smooth map into $G_k(\C^n)$ is harmonic if and only if its composition with the Cartan embedding is harmonic into $\U(n)$}.

For smooth maps $\varphi:M \to G_k(\C^n)$ into a Grassmannian, the operators above are particularly easy to understand.  Note that we can consider such a map as a subbundle, still denoted by $\varphi$, of the trivial bundle $\CC^n$, namely that with fibre at $z \in M$ given by the $k$-dimensional subspace $\varphi(z)$.
Then a short calculation shows that
\begin{equation} \label{second-f-f}
\left\{ \begin{array}{rclcrcll}
A^{\varphi}_z(s)\!\!  &=&\!\! - \pi_{\varphi}^{\perp} \partial_z s & \text{and}  &
D^{\varphi}_z(s)\!\!  &=&\!\!  \pi_{\varphi} \partial_z s, & s \in \Gamma(\varphi), \\
A^{\varphi}_z(s)\!\!  &=&\!\! - \pi_{\varphi} \partial_z s & \text{and} &
D^{\varphi}_z(s)\!\!  &=&\!\!  \pi_{\varphi}^{\perp} \partial_z s, & s \in \Gamma(\varphi^{\perp}).
\end{array} \right.
\end{equation}
Note that $A^{\varphi}_z$ is tensorial (i.e., $A^{\varphi}_z(fs) = fA^{\varphi}_z(s)$ for any smooth function $f$ on $M$), and it maps the subbundle $\varphi$ to its orthogonal complement
$\varphi^{\perp}$ and vice-versa.
In contrast, $D^{\varphi}_z$ acts on sections of $\varphi$ and $\varphi^{\perp}$ as a derivation.
We can give similar formulae for $A^{\varphi}_{\zbar}$ and $D^{\varphi}_{\zbar}$.

Given a harmonic map $\varphi:M \to G_k(\C^n)$, we define its \emph{($\partial'$-)Gauss bundle} or
\emph{Gauss transform}
$G'(\varphi):M \to G_t(\C^n)$ by $G'(\varphi) = \image(A^{\varphi}_z|\varphi)$.
The right-hand side is of some constant rank $t \in \{0,1,\ldots, k\}$ except at some isolated points of $M$: holomorphicity of
$A^{\varphi}_z$ again allows us to fill out zeros giving a rank $t$ subbundle of $\CC^n$, equivalently, a map $G'(\varphi):M \to G_t(\C^n)$.
  \emph{Then $G'(\varphi)$ is also harmonic}, in fact $G'(\varphi)$ is obtained from $\varphi$ by adding the uniton $\varphi + G'(\varphi)$
(cf.\ \eqref{add-uniton}), see \cite[\S 2B]{burstall-wood}.
We iterate this construction to obtain \emph{higher ($\partial'$-)Gauss bundles}
$G^{(i)}(\varphi)$ \ $(i =0,1,\ldots)$ by setting $G^{(0)}(\varphi) = \varphi$, $G^{(1)}(\varphi) = G'(\varphi)$, $G^{(i+1)}(\varphi) = G'(G^{(i)}(\varphi))$.
See \cite{chern-wolfson} for a moving frames approach.

Similarly, we define the \emph{$\partial''$-Gauss bundle} of $\varphi$ by
$G''(\varphi) = \image(A^{\varphi}_{\zbar}|\varphi)$, and iterate this to obtain $G^{(-1)}(\varphi) = G''(\varphi)$, $G^{(-i-1)}(\varphi) = G''(G^{(-i)}(\varphi))$ \ $(i =0,1,\ldots)$, so that $G^{(i)}(\varphi)$ is defined for all $i \in \Z$.  Note that $G'(G''(\varphi) \subset \varphi$ and $G''(G'(\varphi)) \subset \varphi$, often with equality, see
 \cite[Proposition 2.3]{burstall-wood}.  The sequence of Gauss bundles $G^{(i)}(\varphi)$ \ $(i \in \Z)$ is called the
\emph{harmonic sequence of $\varphi$}.

We define the \emph{isotropy order} of a harmonic map $\varphi:M \to G_k(\C^n)$ to be the greatest value of $t \in \{1,2,\ldots,\infty\}$ such that $\varphi$ is orthogonal to $G^{(i)}(\varphi)$ for all $i$ with $1 \leq i \leq t$;
 since $G^{(1)}(\varphi)$ is orthogonal to $\varphi$, there is always such a $t$.  Equivalently \cite[Lemma 3.1]{burstall-wood}, the isotropy order is the greatest value of $t$ such that $G^{(i)}(\varphi)$ and $G^{(j)}(\varphi)$ are orthogonal for $i,j \in \Z$, $1 \leq |i-j| \leq t$.
This divides the harmonic maps into two classes: if $t$ is finite, we say that $\varphi$ is \emph{of finite isotropy order}.
If $t$ is infinite,  i.e., all the Gauss bundles $G^{(i)}(\varphi)$ \ $(i \in \N)$ are mutually orthogonal, we say that $\varphi$ has \emph{infinite isotropy order} or is \emph{strongly isotropic}; this is a particularly nice class of harmonic maps as the following result shows, which also follows from \cite{erdem-wood}.

\begin{proposition} \label{prop:str-iso} All strongly isotropic harmonic maps $\varphi:M \to G_k(\C^n)$ are of finite uniton number.
\end{proposition}

\begin{proof}  Since they are mutually orthogonal, from dimension considerations the Gauss bundles $G^{(i)}(\varphi)$ must be zero for large enough $i$. Thus, after adding a finite number $\alpha_1,\ldots, \alpha_m$ of unitons, we get a constant harmonic map, i.e.,
$\v(\alpha_1-\alpha_1^\perp)\ldots (\alpha_m-\alpha_m^\perp)$ is constant.  This implies that $\v$ has finite uniton number; indeed, it has associated extended solution
\begin{equation}\label{Phi-fact}
\Phi = (\alpha_m^{\perp}+\lambda\alpha_m)\ldots (\alpha_1^{\perp}+\lambda\alpha_1).
\end{equation}
\end{proof}

Here, note that \eqref{Phi-fact} is a uniton factorization.  In general,
it is easy to see that the inverse of adding a uniton $\alpha$ is adding the uniton $\alpha^{\perp}$.

Note that the converse of this proposition holds for $k=1$, see Theorem \ref{th:CPn}, but is false for $k >1$, see Example \ref{ex:G2C4}.

We now turn to the other case, that of finite isotropy order.   We need some notation from \cite{burstall-wood}:
for two mutually orthogonal subbundles $\varphi$ and $\psi$ of $\CC^n$,
we define the \emph{($\pa'$-)second fundamental form of $\varphi$ in $\varphi \oplus \psi$} by
$A'_{\varphi, \psi}(s) = \pi_{\psi} \circ \pa_z s$, \ $s \in \Gamma(\varphi)$; note that this is tensorial. As a special case, we write $A'_{\varphi} = A'_{\varphi, \varphi^{\perp}}$, then \eqref{second-f-f} shows that
$A^{\varphi}_z|{\varphi} = -A'_{\varphi}$ and $A^{\varphi}_z|{\varphi^{\perp}} = -A'_{\varphi^{\perp}}$.

For a harmonic map $\varphi:M \to G_k(\C^n)$ of finite isotropy order $t$, we define the \emph{first return map} $c(\varphi) = c'_t(\varphi)$ \cite{bahy-wood-G2} by
\begin{eqnarray*} \label{first-return}
	c(\varphi) &=& A'_{G^{(t)}(\varphi), \varphi} \circ
	A'_{G^{(t-1)}(\varphi)} \circ \cdots \circ A'_{G^{(1)}(\varphi)} \circ A'_{\varphi}\\
	&=& \pi_{\varphi} \circ A'_{G^{(t)}(\varphi)} \circ
	A'_{G^{(t-1)}(\varphi)} \circ \cdots \circ A'_{G^{(1)}(\varphi)} \circ A'_{\varphi}.
\end{eqnarray*}

We see that the first return map is the composition of the second fundamental forms in the following diagram
where $R:= \bigl(\sum_{i=0}^{t-1} G^{(i)}(\varphi)\bigr)^{\perp}$ and the curved arrow represents $A'_{R, \varphi}$.

\vspace{2ex}
\begin{equation}
\begin{gathered}\label{diag:first-return}
\xymatrixcolsep{3pc}
\xymatrix{
	\varphi \ar[r]_(0.4){A'_{\varphi}} & \ G'(\varphi)
	\ar[r]_{A'_{G^{(1)}(\varphi)}} & \cdots \hspace{-4em} &\hspace{-3em} 
\ar[r]_(0.4){ A'_{G^{(t-2)}(\varphi)}} & G^{(t-1)}(\varphi)\ar[r]_(0.7){ A'_{G^{(t-1)}(\varphi)}} & R
\ar@/_1.2pc/[lllll]
}
\end{gathered} \end{equation}
Further, the maps $A'_{\varphi}, A'_{G^{(1)}(\varphi)}, \ldots, A'_{G^{(t-2)}(\varphi)}$ are surjective and
$A'_{G^{(t-1)}(\varphi)}$ has image $G^{(t)}(\varphi)$, which lies in $R$.
Note that \eqref{diag:first-return} is a \emph{diagram} in the sense of \cite{burstall-wood}, i.e., a directed graph whose vertices (nodes) are mutually orthogonal subbundles $\psi_i$ with sum $\CC^n$; for each ordered pair $(i,j)$, the arrow (i.e., directed edge) from $\psi_i$ to $\psi_j$ represents the second fundamental form $A'_{\psi_i,\psi_j}$; the absence of that arrow indicates that the second fundamental form $A'_{\psi_i,\psi_j}$ is known to vanish.

Then we have the following necessary condition for a harmonic map to be of finite uniton number which was previously known for harmonic maps from the Riemann sphere $S^2$, cf.\ \cite[Proposition 1.9]{burstall-wood}, by using vanishing theorems only applicable in that case.

\begin{theorem}\label{th:first-return} Let $\varphi:M \to G_k(\C^n)$ be a harmonic map of finite uniton number.  Then, either $\varphi$ is strongly isotropic, or its first return map is nilpotent, i.e., $c(\varphi)^k =0$ for some $k \in \N$.
\end{theorem}

\begin{proof}
	Suppose that $\varphi$ is of finite isotropy order $t$.  Write $\psi_i = G^{(i)}(\varphi)$ for $i =0,1,\ldots, t-1$ and $\psi_t = R$ so we have a diagram
\vspace{2ex}	
	\begin{equation}
\begin{gathered}\label{diag:circuit}
\xymatrixcolsep{3pc}
\xymatrix{
	\psi_0 \ar[r]_{A'_{\psi_0}} & \ \psi_1
	\ar[r]_{A'_{\psi_1}}& \cdots \hspace{-4em} &\hspace{-3em} \ar[r]_{A'_{\psi_{t-2}}} &\psi_{t-1} \ar[r]_{A'_{\psi_{t-1}}} & \psi_t
\ar@/_1.2pc/[lllll]
}
\end{gathered}
\end{equation}	
Since $\psi_0 = \varphi$ but the other $\psi_i$ are in $\varphi^{\perp}$, $A^{\varphi}_z$ has components
$A'_{\psi_0} = A'_{\psi_0,\psi_1}$ and $A'_{\psi_t} = A'_{\psi_t,\psi_0}$.  Further, for $0 < i < t$,
$D^{\varphi}_z(\psi_i) = A'_{\psi_i,\psi_{i+1}}(\psi_i)$ \ $\mod \psi_i$.
Hence, from the second formula for $T$ in \eqref{T},
the highest power of $\la^{-1}$ in $\pi_{\varphi}\Ss_{t+1}(T)$ is $\la^{-2}$ which is only achieved by the composition of all the second fundamental forms which make up $c=c(\varphi)$, thus
$\pi_{\varphi}\Ss_{t+1}(T) = \la^{-2} \image c \mod \la^{-1}\HH_+$.  Similarly, for any $k \in \N$,
$\pi_{\varphi}\Ss_{k(t+1)}(T) = \la^{-2k} \image c^k \mod \la^{-2k+1}\HH_+$.
	If $\varphi$ has finite uniton number, by Corollary \ref{S-stabilizes}, the powers of $\la^{-1}$ in $\Ss_i(T)$ are bounded for $i \in \N$; it follows that $\image c^k$ must be zero for $k$ large enough.
\end{proof}

\begin{remark}\rm \label{rem:circuit} Given any diagram \eqref{diag:circuit} of at least three subbundles, by \cite[Proposition 1.6]{burstall-wood}, all the $\psi_i$ are automatically harmonic maps into Grassmannians.  For each $i$, set $c_i$ equal to the composition of second fundamental forms in the cycle
$\psi_i \to \psi_{i+1} \to \cdots \to \psi_t  \to \psi_0 \to \psi_1 \to \cdots \to \psi_i$.
Clearly, \emph{either} all $c_i$ are nilpotent \emph{or} none of the $c_i$ is nilpotent.  Thus, \emph{if \emph{any} $c_i$ is not nilpotent, then \emph{none} of the $\psi_i$ is of finite uniton number}.
\end{remark}

As a first application we show that the converse of Proposition \ref{prop:str-iso} holds for $k=1$; this is well-known for $M = S^2$ \cite{segal}.

\begin{theorem}\label{th:CPn}  Let $\varphi:M \to \CP^{n-1}$ be of finite uniton number.  Then $\varphi$ is strongly isotropic.  In fact,
 $\varphi = G^{(i)}(f)$ for some holomorphic map 	$f:M \to \CP^{n-1}$ and some $i \in \{0,1,\ldots,n-1\}$.
\end{theorem}

\begin{proof} Suppose that $\varphi$ is not strongly isotropic.  Then it has finite isotropy order, say $t$ and we have a diagram \eqref{diag:circuit} in which each subbundle $\psi_i$ is of rank $1$ and equals $G^{(i)}(\varphi)$.
By Theorem \ref{th:first-return} the first return map is nilpotent, but this means
one of the homomorphisms $A'_{\psi_i}$ must be zero.
Hence, $\varphi$ must be strongly isotropic.  It follows that its harmonic sequence
must terminate in both directions, i.e.,  $G^{(-i-1)}(\varphi) = G^{(j)}(\varphi) = 0$ for some $i,j \in \N$.  Taking the least such $i$, set $f = G^{(-i)}(\varphi)$.  Then \cite[Proposition 2.3]{burstall-wood}
$\varphi = G^{(i)}(f)$.  Since the Gauss bundles are mutually orthogonal, for dimension reasons, $i \leq n-1$.
\end{proof}

We next apply our results to an interesting class of maps related to the Toda equations \cite{bolton-pedit-woodward, bolton-woodward}.

\begin{definition} \cite{bolton-pedit-woodward} A harmonic map $\varphi:M \to \CP^{n-1}$ is called \emph{superconformal} if has the maximum possible finite isotropy order.
\end{definition}

Note that, for dimension reasons as in the last proof, the maximum possible finite isotropy order for a harmonic map to $\CP^{n-1}$ is $n-1$.  It follows that a superconformal harmonic map $\varphi:M \to \CP^{n-1}$ has a harmonic sequence which is \emph{cyclic of order $n$} in the sense that
$G^{(n+k)}(\varphi) = G^{(k)}(\varphi)$ for all $k \in \Z$; in particular $G^{(n)}(\varphi) = \varphi$.

We can use Theorem \ref{th:first-return} to give a simple proof of the following:

\begin{proposition} \label{prop:superconf}
A superconformal harmonic map $\varphi:M \to \CP^{n-1}$ is never of finite uniton number.
In fact, $\Ss_{jn}(T) = \la^{-2j}\HH_+$,  equivalently,
$W_{jn} = \la^{-2j}W$, for all $j \in \N$.
\end{proposition}

\begin{proof}
Diagram \ref{diag:first-return} has $R = G^{(n-1)}(\varphi)$, and is cyclic as above; further, all the second fundamental forms in the diagram are isomorphisms almost everywhere.
{}From the second formula for $T$ in \eqref{T} we have
$T(\varphi) = \la^{-1}G^{(1)}(\varphi)$ $\mod \varphi$ and
	$T(G^{(n)}(\varphi)) = \la^{-1}\varphi$ $\mod G^{(n)}(\varphi)$  but
	$T(G^{(i)}(\varphi)) = G^{(i+1)}(\varphi)$ $\mod G^{(i)}(\varphi)$ \ $(i= 1,\ldots, n-1)$.  Thus,
	\begin{eqnarray*}
		\Ss_0 &=& \HH_+, \\
		\Ss_1 &=& \la^{-1}\varphi + \la^{-1}G^{(1)}(\varphi) + \HH_+, \\
		\Ss_j &=& \la^{-2}\sum_{i=1}^{j-1}G^{(i)}(\varphi) + \la^{-1}\sum_{i=0}^j G^{(i)}(\varphi) + \HH_+ \quad
		j=2,\ldots, n-1,\\
		\Ss_n &=& \la^{-2}\sum_{i=1}^n G^{(i)}(\varphi) + \la^{-2}\varphi
		+ \la^{-1}\HH_+ + \HH_+  = \la^{-2}\HH_+.
	\end{eqnarray*}
The formula for $\Ss_{jn}(T)$ follows.	
\end{proof}

The following is a famous example of a superconformal map.

\begin{example} \label{ex:Clifford}
{\rm	For any $n=1,2,\ldots$, consider the smooth map $\varphi:\C \to \CP^{n-1}$ called the
	\emph{Clifford solution} defined in homogeneous coordinates by
	$\varphi = [F]$ where $F=(F_0,\ldots,F_{n-1}):\C \to \C^n$ is given by
	$F_i(z) = (1/\sqrt{n})\,\eu^{\omega^i z - \ov{\omega}^i \ov{z}}$
	with $\omega = \eu^{2\pi\ii/n}$.
Note that this map is harmonic as in Remark \ref{rem:circuit};
it is also isometric, and so defines a minimal immersion, see \cite{jensen-liao}.
	
A simple calculation shows that the $j$th $\pa'$-Gauss bundle, $j \in \Z$, is $G^{(j)}(\varphi) = [F^{(j)}]$  where
	$F^{(j)}_i(z) = (1/\sqrt{n})\,\omega^{ij} \eu^{\omega^i z - \ov{\omega}^i \ov{z}}$; we see that these are orthogonal to $\varphi$ for $j=1,\ldots, n-1$.  However
	$G^{(n)}(\varphi) = \varphi$ so that $\varphi$ is superconformal with harmonic sequence cyclic of order $n$.  Consequently, by Proposition \ref{prop:superconf}, the Clifford solution is not of finite uniton number.
	
When $n=3,4$ or $6$, the Clifford solution factors to an isometric minimal immersion of a torus into $\CP^{n-1}$ called a \emph{Clifford torus}.  When $n=4$
 this torus lies in $\RP^3 \subset \CP^3$ and is double-covered by a minimal torus in $S^3$, again called a Clifford torus, given, up to isometry, by the formula
	$$
	\R^2 \ni (x,y) \mapsto (1/\sqrt{2})(\cos 2x, \sin 2x, \cos 2y, \sin 2y) \in S^3
	$$
which factors to the torus $\R^2/\{\text{lattice generated by } (\pi,0), (0,\pi)\}$.

Given a Clifford solution $\varphi=[F]:\C \to \CP^{n-1}$, let $g(z)$ be the $n \times n$ matrix whose $(j+1)$st column is
$F^{(j)}(z)$; this defines a map $g:\C \to \U(n)$.  Then, by direct calculation $A^g_z\ ( = \frac12 g^{-1}g_z)$ is the constant normal matrix $A$ whose only non-zero entries are
$a_{ij} = 1/2$ when $i-j = 1 \mod n$; further $A^g_{\zbar} = -(A^g_z)^*$.   Define an extended solution $\Phi$ by \eqref{vacuum} with this $A$ and $\tau(z) = z$.  Then
$\Phi$ is a \emph{vacuum solution} as in \S \ref{subsec:constant potentials}. Note that the composition of $g$ with the natural projection $\U(n) \to \U(n)/\U(1) \times \U(n-1) = \CP^{n-1}$ is the Clifford solution $\varphi = [F]$; for this reason the Clifford solution is sometimes also called a vacuum solution.

We remark that $\Phi$ is not an extended solution for the Clifford solution.  In fact, in \cite{aleman-pacheco-wood-symmetry}, we show that  the Clifford solution has extended solution $W = \psi\H_+$ with \emph{constant potential} $\psi^{-1}\pa_z\psi = p(\la) := \la^{-1}A_m + A_k$, where $A_k$ is the part of $A$ in the block diagonal $\U(1) \times \U(n-1)$ and $A_m =  A-A_k$; for example, when
$n=3$,
$$p(\la) = \dfrac{1}{2}\begin{pmatrix}0 & 0 & \la^{-1} \\\la^{-1} & 0 & 0 \\ 0 & 1 & 0\end{pmatrix}.
$$
This gives $W = \exp\bigl((\la^{-1}A_m + A_k)z \bigr)\H_+ = G(\la, z)\H_+$ where
$G:S^1 \times \C \to \U(n)$ is given by
$$G(\la,z) = \exp\bigl((\la^{-1}A_m + A_k)z
	- (\la A_m^{\:*} + A_k^{\:*})\bar{z}\bigr).$$
 This map is an `extended framing' of $\varphi$ and
$\wt\Phi(\la,\cdot) = G(\la,\cdot)G(1,\cdot)^{-1}:M \to \U(n)$ is an extended solution for $\varphi$ with $\wt\Phi(-1,\cdot) = $ (the Cartan embedding of) $\varphi$.
Note that it follows from Theorem \eqref{constant potentials} that harmonic maps associated to $\Phi$ and $\wt\Phi$ are not of finite uniton number; the second of these confirming that the Clifford solution is not of finite uniton number.}
\end{example}

Note that Theorem \ref{th:first-return} implies, as in \cite[Example 4.2]{svensson-wood-twistor-lifts}, that for
a harmonic map of finite uniton number, $A^{\varphi}_z$ is \emph{nilpotent}, cf.\ Corollary \ref{a-1 nilpotent}. Indeed, for a harmonic map of isotropy order $1$, this is equivalent to nilpotency of $c(\varphi)$; for a harmonic map of isotropy order greater than $1$,
$A^{\varphi}_z$ is trivially always nilpotent.

The concept of superconformality can be extended to \emph{harmonic maps into spheres}
$S^{n-1}$.  Indeed, let $\pi:S^{n-1} \to \RP^{n-1}$ be the standard double cover of real projective $(n-1)$-space by the standard round sphere
and let  $i:\RP^{n-1} \to \CP^{n-1}$ be the standard inclusion.  Since the composition $i \circ \pi$  is totally geodesic and isometric, a smooth map $\varphi:M \to S^{n-1}$ is harmonic if and only if $i \circ \pi \circ \varphi:M \to \CP^{n-1}$ is harmonic.  We define the Gauss bundles of $\varphi$ to be those of $i \circ \pi \circ \varphi$ and thus define the isotropy order of $\varphi$, note that this is always odd, see \cite{bolton-pedit-woodward}.  We call a harmonic map $\varphi:M \to S^{n-1}$ superconformal if it is of maximal finite isotropy order \emph{for a harmonic map into $S^{n-1}$}.     There are two cases
\cite{bolton-pedit-woodward,bolton-woodward}:
(i) If $n$ is even, then, as for $\CP^{n-1}$, the maximum finite isotropy order is $n-1$ and such a map has a cyclic harmonic sequence with $i$th vertex $G^{(i)}(\varphi)$ \ $(i \in \N)$; the Clifford torus in $S^3$ above is of this type.
(ii) If $n$ is odd, then the maximum finite isotropy order is $n-2$ and diagram \eqref{first-return} has one vertex of dimension $2$ but the first return map is still an isomorphism;
this case includes non-full harmonic maps into an equatorial $S^{n-2}$ in $S^{n-1}$, described as in case (i).

In particular, again Theorem \ref{th:first-return} provides a simple proof of the following.

\begin{corollary} \label{cor:superm-sphere}
A superconformal harmonic map from a Riemann surface to a sphere is never of finite uniton number.
\end{corollary}

The following example shows that the converse to Theorem \ref{th:first-return} is false.

\begin{example} \label{ex:superconf}
{\rm	Let $\psi:M \to \CP^{n-1}$ be a superconformal map with $n \geq 5$;
	write $\psi_i = G^{(i)}(\psi)$ and set $\varphi = \psi_0 \oplus \psi_2$.  Then
	$\varphi:M \to G_2(\C^n)$ is a harmonic map: indeed the harmonicity equation \eqref{harmcond} is easily seen to be the sum of those for $\psi_0$ and $\psi_2$.

(Two other ways to see this are that $A'_{\varphi}$ consists of the second fundamental forms
	$A'_{\psi_0,\psi_1}$ and $A'_{\psi_2,\psi_3}$ which are holomorphic by \cite{burstall-wood}.  Or note that $\varphi$ has a `$J_2$-holomorphic twistor lift' $(\psi_0,\psi_1,\psi_2,\sum_3^n \psi_i)$ and so is harmonic by \cite{svensson-wood-unitons}.)
	
We see that $A'_{\varphi^{\perp}} \circ A'_{\varphi}$ has image
	$\psi_2$ so $\varphi$ has isotropy order $1$ with first return map $c(\varphi) = A'_{\varphi^{\perp}} \circ A'_{\varphi} = (A_z^{\varphi})^2|{\varphi}$.  This is nilpotent, in fact $c(\varphi)^2=0$.  However, calculating as in the proof of
Proposition \ref{prop:superconf}, we see that $\Ss_{jn}(T) = \la^{-4j}\HH_+$,
 equivalently, $W_{jn} = \la^{-4j}W$, so that \emph{$\varphi$ has nilpotent first return map but is not of finite uniton number}.

The above works for any domain $M$, compact or not; by taking $\psi$ to be the Clifford torus in
$\CP^5$, we get an example with compact domain.}
\end{example}

We now consider an arbitrary diagram, where a harmonic map
$\varphi:M \to G_k(\C^n)$ is the sum of the subbundles represented by some of the vertices.  Recall that the second fundamental forms between the various subbundles are represented by the arrows (directed edges) between the vertices.
Call an arrow \emph{external} if one of its vertices is in $\varphi$ and the other is in $\varphi^\perp$. A
 path is \emph{external} if it contains at least one external arrow.  A \emph{simple cycle of length $p$ for $\psi_0$} is a sequence of arrows $\psi_i \to \psi_{i+1}$, $i=0,1,\ldots, p-1$, with $\psi_0=\psi_p$ and $\psi_0,\psi_1,\ldots, \psi_{p-1}$ distinct vertices. The composition of second fundamental forms making up an external simple cycle $c$ for a vertex $\psi_i$ defines a linear endomorphism of $\psi_i$ which we also denote by $c$.   We have a test which follows from the bounded powers criterion:

\begin{theorem}\label{th:external}
\begin{enumerate}
\item[(i)] Suppose that there exist no external cycles.  Then $\varphi$ is of finite uniton number.
\item[(ii)] Suppose that, for some vertex, there is a \emph{unique} simple external cycle with $c$ not nilpotent.  Then $\varphi$ is not of finite uniton number.
\end{enumerate}
\end{theorem}
\begin{proof} (i)  For each path $w$, let $|w|_e$ be the number of external arrows in $w$. Since there exist no external cycles,  $|w|_e$ is bounded for an arbitrary path $w$. Hence, the sequence $\Ss_i(T)$ stabilizes, since the increasing of the power of $\lambda^{-1}$ is exclusively due to the contribution of the external arrows. We conclude, by Corollary \ref{S-stabilizes}, that $\varphi$ is of finite uniton number.

(ii) Suppose that for the vertex $\psi_i$ there exists a unique simple external cycle $c$.  Let $p$ be the length of the simple cycle and $|c|_e$ be the number of external arrows in $c$. Then
$\la^{-|c|_e}\image c\subseteq \Ss_{p}$ and, for any $k>0$, $\la^{-|c|_ek}\image c^k\subseteq \Ss_{pk}$. Consequently, if $c$ is not nilpotent, the sequence $\Ss_i(T)$ does not stabilize and $\varphi$ is not of finite uniton number.
\end{proof}

We note that Proposition \ref{prop:str-iso} and Theorem \ref{th:first-return} follow from parts (i) and (ii) of this result.  Here is another application of part (i). Say that a harmonic map $\varphi:M \to G_k(\C^n)$ \emph{has a $J_1$-holomorphic twistor lift $(\psi_0,\ldots,\psi_d)$} if the $\psi_i$ form a diagram, the second fundamental forms $A'_{\psi_i, \psi_j}$ with $j \leq i$ are zero, and
$\varphi = \sum \psi_{2i}$; see, for example, \cite{svensson-wood-twistor-lifts}.  Since such a diagram has no external cycles we have:

\begin{corollary} \label{cor:no-loops}
All harmonic maps $\varphi:M \to G_k(\C^n)$ with a $J_1$-holomorphic twistor lift are of finite uniton number.
In particular, all harmonic maps with an $S^1$-invariant extended solution are of finite uniton number.
\end{corollary}

The last part follows as such a harmonic map has a $J_1$-holomorphic twistor lift $(\psi_0,\ldots,\psi_m)$ constructed from the unitons $\alpha_i$ of
\eqref{superhor} by $\psi_i = \alpha_{i+1} \ominus \alpha_i$ (where we take $\alpha_0$ to be the zero vector bundle and $\alpha_{m+1}$ to be the trivial bundle $M\times \C^n$).

We give an example illustrating that \emph{(a)} part (ii) of the proposition is false without the word `unique'; \emph{(b)} we cannot, in general, tell whether the uniton number of a harmonic map into a Grassmannian is finite just from a diagram.

\begin{example} \label{ex:G2C4}
\rm {\rm (a)}	As in \cite[Example 8.3]{svensson-wood-twistor-lifts} (with a slight change of notation), for any Riemann surface
	$M$, set
	\begin{equation} \label{W-G2C4}
	W=\spa\{H +\lambda^2 K\} + \lambda h_{(1)}  +  \lambda^2 h_{(2)} +  \lambda^3\H_+
	\end{equation}
	where $H, K:M \to \C^4$ are meromorphic maps (equivalently, meromorphic sections of the trivial bundle
	$\CC^4$), with $H$ \emph{full}, i.e., not lying in a proper trivial subspace of\/ $\C^4$, and $h$ is the holomorphic map $M \to \CP^3$, equivalently, the rank $1$ holomorphic subbundle of\/ $\CC^4$,
 given by $h = \spa\{H\}$.
For $i=1,2,\ldots$, let $h_{(i)}$ denote the $i$th osculating subbundle of\/ $h$, given by the span of\/ $H$ and its derivatives of orders up to $i$, cf.\ \eqref{Gauss}; the corresponding holomorphic map $h_{(i)}:M \to G_t(\C^n)$ ($t \in \{0,1,\ldots, i+1\}$) is called the \emph{$i$th associated curve} ---
note that the Gauss bundle $G^{(i)}(h)$ equals $h_{(i-1)}^{\perp} \cap h_{(i)}$.
It can be checked that $W$ satisfies the conditions \eqref{W-eq}; in fact,
by the formulae in \cite{svensson-wood-unitons}, $W = \Phi\H_+$ for the extended solution given by the product of unitons:
	\begin{equation} \label{Uhl-fact}
	\Phi = (\pi_{h_{(2)}} + \la \pi_{h_{(2)}}^{\perp})(\pi_{h_{(1)}} + \la \pi_{h_{(1)}}^{\perp})(\pi_{\psi_0} + \la \pi_{\psi_0}^{\perp})
	\end{equation}
	where $\psi_0 = \spa\{H + \pi_{h_{(2)}}^{\perp}K\}$.
 Then $\varphi:= \Phi(-1,\cdot)$ is a harmonic map $M \to \U(n)$ of finite uniton number.
Since $W$ satisfies the symmetry condition $W_{\la} = W_{-\la}$, this harmonic map actually has values in a Grassmannian (cf.\ \cite[\S 5.1]{svensson-wood-unitons}); on putting $\la=-1$ in \eqref{Uhl-fact}, we see that $\varphi = $ (the Cartan embedding of) $ \psi_0 \oplus \psi_2:M \to G_2(\C^4)$ where $\psi_2 = G^{(2)}(h)$.  Set $\psi_1:= G^{(1)}(h)$. This is orthogonal to $\varphi$ and so we have $\varphi^{\perp} = \psi_1 \oplus \psi_3$ where $\psi_3$ is the orthogonal complement of\/ $\psi_0 \oplus \psi_1 \oplus \psi_2$ in $\CC^4$.

By \cite{svensson-wood-twistor-lifts}, the only possible non-zero (i.e., not identically zero) second fundamental forms are those given by the diagram:
	\begin{equation}
	\begin{gathered}\label{diag:G24}
	\xymatrixrowsep{1.8pc}\xymatrix{
		\psi_0  \ar[rd] \ar[r] &\psi_1 \ar[ld] \\
		\psi_2 \ar[r]  \ar[u] &\psi_3 \ar[u]}
	\end{gathered}
	\end{equation}
	
(In fact, $Z_i :=   (A^{\varphi}_z)^i(\C^4)$, $i=0,1,\ldots$, gives a filtration of $\CC^4$ with $(A^{\varphi}_z)(Z_i) = Z_{i+1}$; then $\psi_i = Z_{i+1}^{\perp} \cap Z_i$, which leads to this diagram.)

Note that $W$ is $S^1$-invariant if and only if $K$ lies in $h_{(2)}$; equivalently, $\psi_0 = h$.  In that case,
$\psi_i = h_{(i)}$, $i=0,1,2,3$, and the second fundamental forms other than $A'_{\psi_i,\psi_{i+1}}$, $i=0,1,2$, are zero.
Now, fullness of $h$ implies that these second fundamental forms are non-zero in all cases. Further, a simple calculation shows that, when $W$ is not $S^1$-invariant, i.e., $K$ does not lie in $h_{(2)}$,
 $A'_{\psi_2,\psi_0}$ is also non-zero.
So we have a cycle of non-zero second fundamental forms $\psi_0 \to \psi_1 \to \psi_2 \to \psi_0$.  Since the first two arrows in this cycle are external, i.e., components of $A^{\varphi}_z$, following that cycle shows that
	\emph{$\Ss_3(T)$ contains a non-zero multiple of $\la^{-2}\psi_0$.  However, $\Phi$ and so $\varphi$, is of finite uniton number by construction}.
	Note that this does not contradict Theorem \ref{th:external}(ii) as there are other cycles on $\psi_0$ with $c$ not nilpotent, for example,
	$\psi_0 \to \psi_3 \to \psi_1 \to \psi_2 \to \psi_0$.
	
{\rm (b)} In contrast, consider a superconformal harmonic map $\varphi:M\to\CP^{3}$.  This has Gauss sequence:
\vspace{1ex}
\begin{equation}
\begin{gathered}
\xymatrix{
	 \varphi \ar[r] & \varphi_1 \ar[r]
	 & \varphi_2 \ar[r] & \varphi_3 \ar@/_1.0pc/[lll]}
\end{gathered}
\end{equation}
Then, $\psi=\varphi_2\oplus \varphi_3$  and $\psi^\perp=\varphi\oplus \varphi_1$ are harmonic (cf.\ \cite[Proposition 1.6]{burstall-wood}). Choose a holomorphic line subbundle $R_h\neq \varphi$ of $\psi^\perp$. Its orthogonal complement $R_a$ in $\psi^\perp$ is antiholomorphic in $\psi^\perp$. For example, if we consider the Clifford solution $\varphi=[F]:\C\to \CP^3$, as defined in Example \ref{ex:Clifford},  then we can choose the holomorphic subbundle $R_h=[\bar z F+F^{(1)}]$, whose orthogonal complement is $R_a=[F-zF^{(1)}] $. The subbundles $\alpha=\varphi_2$ and $\beta= R_h$ satisfy the  forward replacement conditions cf.\ \cite[Theorem 2.4]{burstall-wood} with respect to $\psi$,
so that $\alpha \oplus \beta$ is a uniton. Hence,
$\varphi=\varphi_3\oplus R_h$ is harmonic.
Set
$$
\psi_0=R_h,\quad \psi_1=\varphi_2, \quad \psi_2=\varphi_3, \quad \psi_3=R_a.
$$
The corresponding diagram coincides with \eqref{diag:G24}, with all arrows non-vanishing, but now the harmonic map $\varphi$ cannot be of finite uniton number because $\psi$ is not.
\end{example}

\end{document}